\documentclass[12pt]{article}

\usepackage{color,graphicx}

\usepackage[T2A]{fontenc}
\usepackage[utf8]{inputenc}
\usepackage[russian,english]{babel}

\usepackage{fullpage}
\usepackage{amsmath,amssymb,amsthm}
\usepackage{epsfig}\usepackage{graphicx}
\newtheorem{theorem}{Theorem}
\newtheorem{lemma}{Lemma}

\newtheorem{prop}{Proposition}

\theoremstyle{remark}
\newtheorem{ex}{Example}
\newtheorem{remark}{Remark}
\newtheorem{definition}{Definition}

\newcommand{\im}{\text{Im\,}}
\newcommand{\N}{\mathbb N}
\newcommand{\R}{\mathbb R}
\newcommand{\pair}[1]{\langle #1\rangle}

\newcommand{\cen}{c}
\newcommand{\crown}{\textit{Crown}}

\newcommand{\eps}{\varepsilon}
\newcommand{\types}{\textit{Typ}}
\newcommand{\marked}{\mu}
\newcommand{\undef}{\bot}

\date{}

\title{Matching Rules for Substitution and 
Hierarchical Tilings for any Substitution with Finite Local Complexity}
\author{Nikolay Vereshchagin\\
  Moscow State University, HSE University,
  Yandex\thanks{This paper was prepared within the framework of the HSE University Basic Research Program (project no.~HSE-BR-2025-24).}}

\begin{document}
\maketitle

\begin{abstract}
The Goodman-Strauss theorem states that for ``almost every'' substitution $\tau$, 
the family of substitution tilings is sofic, that is, it can be defined by local matching rules 
for some decoration of tiles. The 
conditions on the substitution that guarantee the soficity 
are quite complicated in the statement of the theorem.  
In this paper we propose a version of the Goodman-Strauss theorem with very simple conditions on the substitution:
the family of substitution tilings must have finite local complexity (FLC), that is, 
the number of crowns that appear in $\tau$-supertiles is finite.
Like the original theorem, our theorem provides matching rules for 
all known substitution tilings. 

We also prove a similar theorem for the family of \emph{hierarchical} tilings associated with the given substitution. 
A tiling is called $\tau$-hierarchical if it has a composition under $\tau$, such that this composition also has a composition, 
and so on, infinitely many times. Every substitution tiling is hierarchical, but the converse is not always true. 
\end{abstract}
\section{Introduction}

\subsection{Substitutions and Substitution Tilings}

Assume that a  finite set of polygons $P=\{\alpha_1,\dots,\alpha_M\}$ is given. 
Some of them may be congruent; to distinguish them, we  
mark them with different colors. These polygons are called \emph{prototiles}, and \emph{tiles} are any shifts of prototiles.
Prototiles cannot be rotated or flipped.
A shift of $\alpha_i$ is called a tile \emph{of the form $\alpha_i$}.
A \emph{tiling} is any set of pairwise non-overlapping  tiles, that is, tiles having no common inner points. 
We will say that line segments \emph{overlap} if they share a fragment of positive length. 
 A tiling is said to be \emph{side-to-side} 
if any two overlapping sides of its  tiles coincide. 

In addition,
a \emph{substitution} $\tau$ is given, that is, for each
polygon $\alpha_i$ a \emph{rule} $\alpha_i\to\mathcal{M}_i$ is given,
where $\mathcal{M}_i$ is some tiling of the polygon $\theta \alpha_i$.
Here $\theta>1$ is a real number  and multiplication by $\theta$
means stretching by $\theta$ times without rotation.
The tiling $\mathcal{M}_i$
is called 
 the \emph{$\tau$-decomposition} of $\alpha_i$
and is denoted by $\tau \alpha_i$. The image of a side $a$ of $\alpha_i$ 
is denoted by $\tau a$, those images 
are called \emph{macrosides}.
The tiles from $\mathcal{M}_i$ are called \emph{children}
of the tile $\alpha_i$, and the tile itself is their \emph{parent}.
Two examples of a substitution are shown in Fig.~\ref{f13}, \ref{f14}.
\begin{figure}[t]
\begin{center}
\includegraphics[scale=.6]{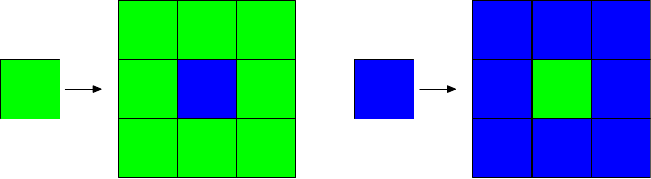}
\end{center}
\caption{First  substitution}\label{f13}
\end{figure}
\begin{figure}[t]
\begin{center}
\includegraphics[scale=.4]{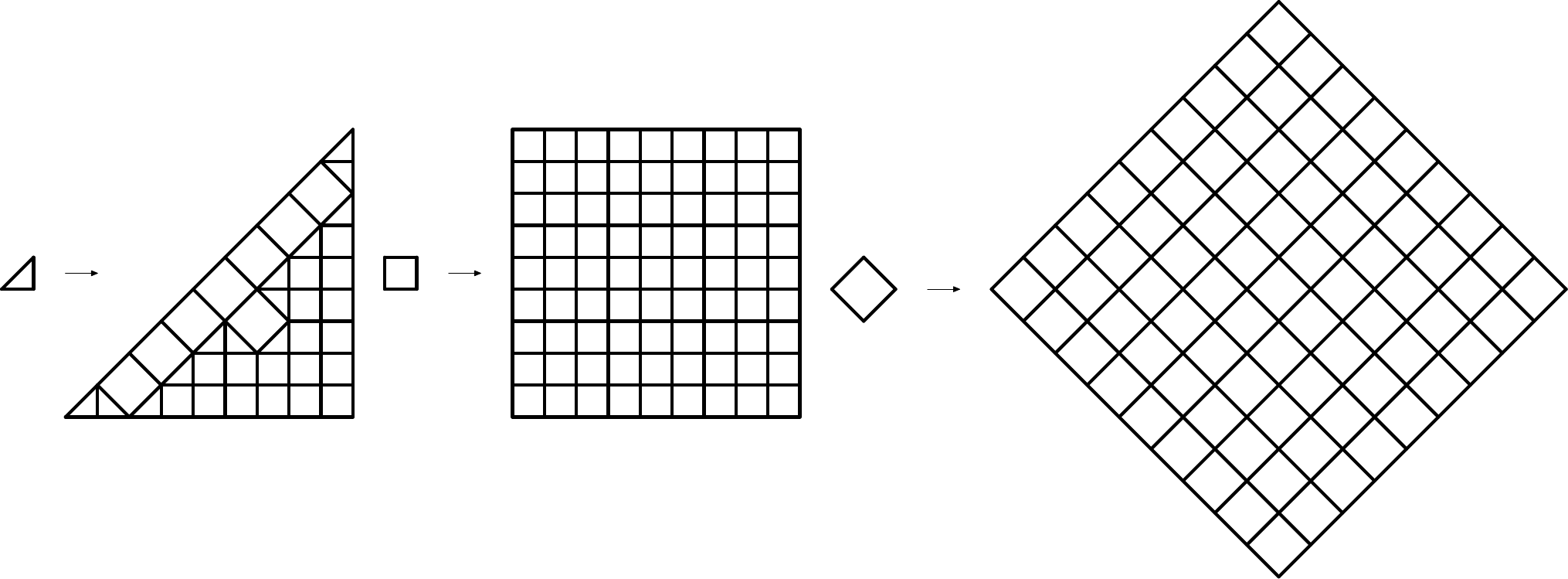}
\end{center}
\caption{Second substitution. The first prototile can be rotated by 90, 180, and 270 degrees.
Therefore, the total number of prototiles is 6.}\label{f14}
\end{figure}

The action of substitution on tilings is defined as follows: the
rules are applied to all tiles of the tiling simultaneously, i.e., the tile $\alpha_i$ is replaced by the tiling $\mathcal M_i$.
The resulting tiling is called the \emph{decomposition} of the original tiling.
The inverse operation is called the \emph{composition}, the composition of a tiling
may not exist and may not be unique.

\emph{Supertiles} are finite tilings obtained from prototiles by several 
decompositions. The number of decompositions is called the \emph{order} of the supertile.
For example, all macrotiles are supertiles of the first order.
See  Fig.~\ref{supertiles} for another example.
\begin{figure}[t]
\begin{center}
\includegraphics[scale=.8]{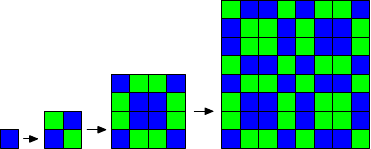}\hskip 2cm \includegraphics[scale=.8]{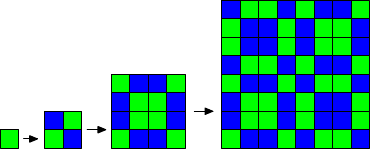}
\end{center}
\caption{The figure shows a substitution and supertiles of order 1, 2 and 3.}\label{supertiles}
\end{figure}
A tiling is called \emph{a $\tau$-substitution tiling}
if each of its finite parts is included in some $\tau$-supertile.

In general, a \emph{local rule} for tilings with tiles from $P$ 
is identified by 
a set $\mathcal R$ (possibly infinite) of tilings with $P$-tiles such that for some real number $D$ 
the diameter of all
tilings in $\mathcal R$  is at most $D$.
We say  that a tiling $\mathcal T$ 
satisfies the local rule $\mathcal R$, or is \emph{proper} (if  $\mathcal R$ is clear from the context),
if it has no subset whose shift belongs to $\mathcal R$.
Tilings from $\mathcal R$ are called \emph{illegal}.

A \emph{decoration}  of a tile set $P$ is any tile set $P'$
and a surjective function  $\pi:P'\to P$, called 
the \emph{projection}, such that 
geometrically $\pi\alpha$ coincides with $\alpha$ for all $\alpha\in P'$,
that is, $\pi\alpha$ is a shift of $\alpha$.
We may think that tiles whose projections coincide have different colors.

A family of tilings $\mathcal F$  is called   \emph{sofic} if 
there is a decoration $P'$ of $P$ and a local rule $\mathcal R$ for $P'$-tiles such that 
a tiling  is in $\mathcal F$ 
if and only if it is the projection of a proper tiling with tiles
from $P'$.

\begin{ex}
Here is an example of decoration and local rule.
For every $\alpha\in P$ we consider several shifts of $\alpha$ and for each shift we color
each side of that shift in a certain color.
The projection then maps all shifts of $\alpha$ back to $\alpha$.
The local rule stipulates that tiles can be connected only side-to-side so that
colors of shared sides match.
This is the most common kind of decoration and local rules.
Such decorations will be called \emph{side decorations}. 
If it is possible to prove that $F$ is sofic via a decoration and a local rule of this kind, 
then we call $F$ \emph{side-to-side sofic}. Obviously all tilings from a side-to-side sofic family are side-to-side.
\end{ex}

\subsection{Mozes' Theorem}
Mozes' theorem~\cite m states that under certain conditions on the substitution
the family of substitution tilings is side-to-side sofic.
The assumptions on the substitution in Mozes' theorem are that all prototiles are squares of the same size 
and one more condition that we find difficult to state here.

\subsection{Goodman-Strauss Theorem}
The Goodman-Strauss theorem generalizes Mozes' theorem. It
states the same thing (the soficity of
the family of substitution tilings) but under weaker conditions
on the substitution.
Here is how this theorem is formulated in~\cite{gs}:
\emph{Every substitution tiling of $\R^d$, $d > 1$, can be enforced with finite
matching rules, subject to a mild condition: the tiles are required to admit a set of ``hereditary sides''
such that the substitution tiling is ``sibling-side-to-side''}. In this quotation, the phrase
``Every substitution tiling of $\R^d$, $d > 1$, can be enforced with finite matching rules'' in our terminology means: ``Every family of substitution tilings is sofic''.
But the second part of the formulation, which talks about sufficient conditions for this,
is quite complicated, see Section 1.4 of~\cite{gs} on pp.  181--182. 

\subsection{Fernique --- Ollinger construction}

A tiling $\mathcal{T}$
is called \emph{[side-to-side] $\tau$-hierarchical} 
if there exists an infinite sequence of
[side-to-side]
tilings
$\mathcal{T}_0=\mathcal{T},\mathcal{T}_1,\mathcal{T}_2,\dots$ in which for all $i$ the tiling $\mathcal{T}_{i+1}$ is a $\tau$-composition of $\mathcal{T}_i$.

The Fernique --- Ollinger theorem~\cite{fo} states that, under some conditions on the given substitution, the set of side-to-side-hierarchical
tilings is side-to-side sofic.\footnote{In fact, their statement is more general, they consider a broader class of substitutions
--- the so-called ``combinatorial substitutions''. Also they consider tilings of $\R^d$ for $d>2$.} 

\subsection{Our contribution}

Let us consider the following  assumptions called STS.
\begin{enumerate}
\item[STS:]\label{r0}
  All $\tau$-supertiles are side-to-side tilings.  (Hence
  every $\tau$-substitution tiling is side-to-side.) 
\end{enumerate}

Our contribution consists of the following theorems.
\begin{enumerate}
\item We improved the Fernique --- Ollinger construction
to work 
under the STS assumption.
Namely, we proved that the family of side-to-side
hierarchical tilings
is  side-to-side sofic provided the given substitution satisfies STS.  

\item
Under STS
we proved that the  family of substitution  tilings
is side-to-side sofic as well.  

\item 
We proved that 
both families of hierarchical and substitution tilings are 
sofic provided  the given substitution  has finite local complexity. Let us define what this means.
\begin{definition} \label{cr-all}
Let $\mathcal{T}$ be a tiling  with $\tau$-tiles.
  \emph{A crown} of  $\mathcal{T}$  at a vertex $V$ of its tile is the fragment
  of $\mathcal{T}$ consisting of all tiles
  from $\mathcal{T}$ that include $V$. 
  A crown is called \emph{$\tau$-allowed} if it is a crown at some vertex of some $\tau$-supertile.
 We say that $\tau$ \emph{has Finite Local Complexity (FLC)} if 
the number of $\tau$-allowed crowns is finite (up to shifts). 
\end{definition}
Note that STS implies that $\tau$ has FLC, as the total number of crowns that are side-to-side tilings is finite.\footnote{
Indeed, in a side-to-side tiling no vertex of a tile lies in the interior of a side of another tile,
so a crown at a vertex $V$ consists of tiles surrounding $V$ side-to-side; the number of such
	configurations up to shifts is bounded in terms of the prototiles.}
But not the other way around. Thus,
compared to the first two results, here both the assumption and the conclusion are weaker. 
\end{enumerate}

A previous version of our technique that works under quite general but 
much more complicated and restrictive assumptions was presented in the previous 
paper~\cite{ver} by the author.

\section{The relationship between substitution and hierarchical tilings}

\begin{lemma}\label{th11}
Every substitution tiling $\mathcal{T}$ of the plane
 has a composition that is
again a substitution tiling.
Hence every substitution tiling is hierarchical.
\end{lemma}
\begin{proof}
Let $\tau$ be a substitution and $\mathcal{T}$ a substitution tiling.
Call any supertile $S$ for which
$F\subset\tau S$
a \emph{cover} of a tiling $F$.
Any finite $F\subset \mathcal T$ has a cover. Indeed,
$F$ is included in some supertile and
the composition of that supertile is a cover of $F$.

Let
$A_1,A_2,\dots$  be an enumeration of all tiles from $\mathcal T$.  For each $n$
let $S_n$ denote  any cover  of the set $\{A_1,\dots, A_n\}$.
For $i\le n$, we call the tile $B\in S_n$ for which $A_i\in \tau B$  the \emph{parent of $A_i$ in $S_n$}. For $m>n\ge i$, the tile
$A_i$ may have different parents in $S_m$ and in $S_n$.
However, by removing some terms from the sequence $S_1,S_2,\dots$, we can
ensure that this does not happen, namely, that for all $m> n$
the parents of $A_n$ in $S_m$ and $S_n$ coincide.

This is done using a diagonal construction. First, note that
for any tile $A_i$, the set of all possible parents of $A_i$ is finite. Indeed, a parent of $A_i$
can be identified by its form and the location of  $A_i$ in its decomposition.

Now we choose any tile $B_1$ such that for 
infinitely many $i$ the tile $B_1$ is the parent  of $A_1$  in
$S_i$. Remove from the sequence
$S_ 1,S_2,\dots$ all tilings $S_i$ for which the parent of $A_1$ in
$S_i$ is different from $B_1$. Denote by
$S'_1,S'_2,\dots$
the resulting infinite sequence of tilings.
Fix the first member $S_1'$ in it, and
thin out the sequence $S'_2,S'_3,\dots$ so that $A_2$ has
the same parent in all its tilings.
Denote by
$S''_2,S''_3,\dots$
the resulting infinite sequence. And so on. The sought sequence is $S'_1,S''_2,S'''_3,\dots$

So, we can
assume that for all $m> n$
the parents of $A_n$ in $S_m$ and $S_n$ coincide.
Now we can construct a composition of the tiling $\mathcal T$.
This is the set
$$
\mathcal T'=\{\text{the parent of }A_n\text{ in }S_n\mid n\in\N\}.
$$
By construction, the decomposition of this set contains
all the tiles $A_1,A_2,\dots$. It remains for us to prove
that $\mathcal T'$ is a substitution tiling,
in particular, different tiles from $\mathcal T'$ do not overlap.

It suffices to prove that for all $n$ the set
of tiles
$$
\{\text{the parent of }A_1 \text{ in }S_1,
\text{the parent of }A_2 \text{ in }S_2,\dots,
\text{the parent of }A_n \text{ in }S_n\}
$$
is included in some supertile.
By construction, this set
coincides with the set
$$
\{\text{the parent of }A_1 \text{ in }S_n,
\text{the parent of }A_2 \text{ in }S_n,\dots,
\text{the parent of }A_n \text{ in }S_n\},
$$
which is included in the supertile $S_n$.
\end{proof}

The converse is not true in general: for some substitutions there are side-to-side-hierarchical tilings that are not 
substitution tilings. 
Note that under STS 
all substitution tilings are side-to-side.
Hence every substitution tiling is side-to-side-hierarchical (under STS).

To construct local rules for substitution tilings we need some sufficient
conditions for the converse statement to hold.

\begin{lemma}\label{l22}
  Assume that a sequence of tilings
  $\mathcal{T}_0=\mathcal{T}, \mathcal{T}_m, \mathcal{T}_{2m}, \dots$
  witnesses that the tiling   $\mathcal{T}$ is $\tau^m$-hierarchical. Here $m$ is any positive integer.
  Assume further that $\tau$ has FLC and that
  all crowns in all tilings
  $\mathcal{T}_{im}$ are $\tau$-allowed. 
  Then $\mathcal{T}$ is a $\tau$-substitution tiling. 
\end{lemma}
\begin{proof} We start with the following
\begin{lemma}\label{l7}
Assume that $\tau$ has FLC. Then there is $\eps>0$ (depending on $\tau$) with the following properties. 
(a) Let  $\mathcal T$ be a tiling of the plane with tiles from $P$
all of whose crowns are $\tau$-allowed.
Assume further that $S$ is a set on the plane 
of diameter less than $\eps$. 
Then $\mathcal T$ has a crown that covers  $S$.
(b) A similar statement holds when $\mathcal T$ is a macrotile. This time we claim that $\mathcal T$ has a crown that covers  $S\cap R$
where $R$ stands for the union of all tiles in $\mathcal T$.
\end{lemma}
The proof can be found in the Appendix.

 Let $F$ be an arbitrary finite fragment of $\mathcal{T}$. We have to  prove that it is included in some supertile.
 Let $A$ be a tile from $\mathcal{T}_{km}$. Then its $mk$-fold decomposition is a supertile of
order $km$ included in  $\mathcal{T}$. The length of its supersides is $\theta^{km}$
times larger than those of $A$. By Lemma~\ref{l7}(a) there is a crown in  $\mathcal{T}_{km}$ whose
$mk$-fold decomposition covers and hence includes $F$, provided $\eps\theta^{km}$
is larger than the diameter of $F$.

By assumption
that crown is allowed, that is, it is contained in some supertile $S$.
  It follows that $F$ is contained in the $km$-fold decomposition of the supertile $S$,
  which is also a supertile. Notice that the value of $k$
depends only on $\tau$ and the diameter of $F$. 
\end{proof}

\section{Under STS assumption}
\begin{theorem}\label{th22}
  Assume that all $\tau$-supertiles are side-to-side tilings (assumption STS).
  Then (a)  the family of  $\tau$-substitution tilings  is side-to-side sofic, (b)
  the family of $\tau$-side-to-side-hierarchical tilings is  side-to-side sofic,
and (c) the family of $\tau$-hierarchical tilings is  sofic.
(Later we will prove item (c) under a weaker assumption.)
\end{theorem}
In the rest of this section we prove this theorem.
We start with the proof of (a).
We will construct our set of decorated tiles in four steps:
\begin{enumerate}
\item First, we add to $P$ some prototiles obtained by joining
some original tiles, and denote the resulting set by $P'$.

\item Second, we decorate the tiles of $P'$, obtaining a set of colored tiles $P''$.
On the tiles of this set, we consider a certain substitution $\sigma''$,
which will act as some power $\tau^m$ of the original substitution,
ignoring the decoration of  tiles.

\item We then remove some tiles
from $P''$, obtaining the set of decorated tiles $P'''$.
\item Finally, we dissect each tile of $P'''$ that was added during the transition from $P$ to $P'$ back into its original tiles, coloring the new sides by new unique colors.
\end{enumerate}

\subsection{The tile set $P'$}\label{s3.1}
Let $P=\{\alpha_1,\dots,\alpha_M\}$ and $\tau$ be the given prototile set and substitution.
We add some new prototiles to $P$, each of which is the union of the original
prototiles. Then we define some substitution $\sigma'$ on it.
We will use $\Delta$ to denote the operation of \emph{scattering} the added tiles:
$\Delta \alpha$ for $\alpha\in P'$ is the tiling consisting of the tiles that make up
$\alpha$.
The substitution $\sigma'$ will be \emph{compatible} with some power
$\tau^m$ of the original substitution, in the following sense: for any tile $\alpha\in P'$,
$\Delta\sigma'\alpha=\tau^m\Delta\alpha$. That is, we have $\sigma'\alpha=\Delta^{-1}\tau^m\Delta\alpha$ in a sense.

To do this, consider a sufficiently large power $\tau^k$ of the original substitution (how large will become clear later).
In each $\tau^k$-macrotile, strictly inside
it, we mentally draw a grid of a sufficiently large number $(N+2)\times (N+2)$ of sufficiently large identical squares.
The grid for $N=6$ is shown in green in
Fig.~\ref{f418}.
The side length $R$ of the squares must be significantly
larger than the maximal size of the prototiles.

From the central $N\times N$ squares we form one \emph{large square}. 
\begin{figure}
\begin{center}
\includegraphics[scale=.9]{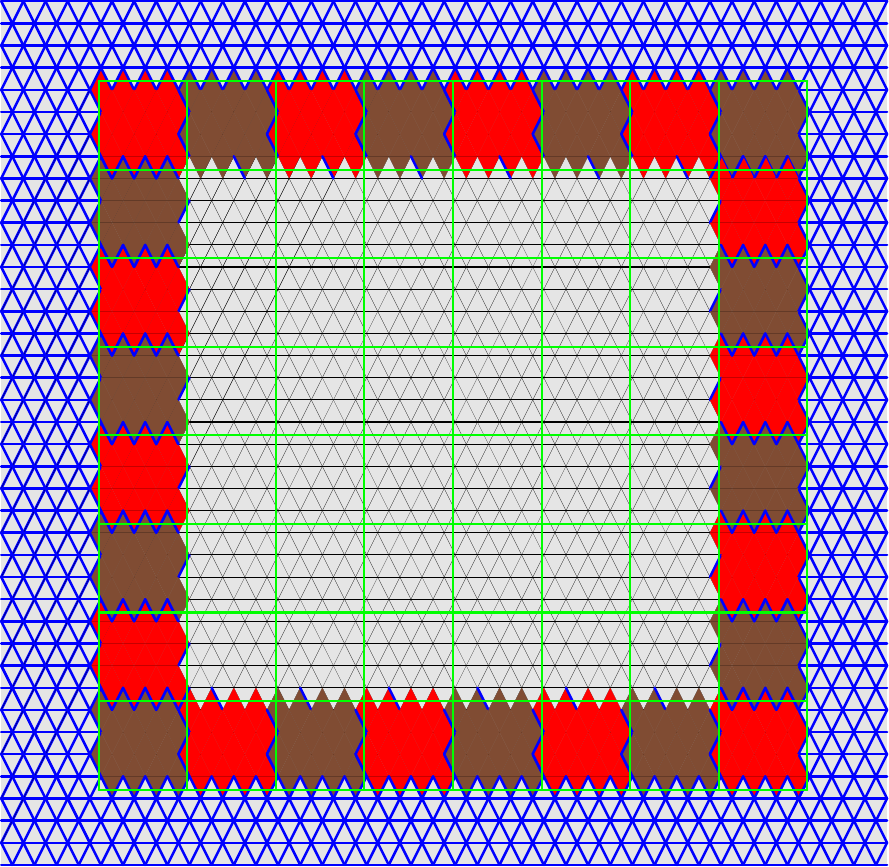}
\end{center}
\caption{The central part of a macrotile consisting of triangular tiles.
The large quasi-square is grey. Small quasi-squares are red and brown. 
Marked sides are blue (all sides except some sides of the large quasi-square).}\label{f418}
\end{figure}
We assign each tile that shares an inner point with at least one of the squares
to any of those squares. For each square, we join all  the tiles  assigned to that square. The resulting
tiles will be called \emph{quasi-squares}. 
We consider each of the quasi-squares to be a single new prototile.

In Fig.~\ref{f418}, the small quasi-squares are colored red and brown. Inside the ring of
small quasi-squares is the large quasi-square. Outside the ring of
small quasi-squares are the original tiles.

Thus, the tile set $P'$ consists of the original prototiles, small quasi-squares, and large
quasi-squares. In total, the set contains $M$ original prototiles, $M$ large
quasi-squares, and $M((N+2)^2-N^2)=4M(N+1)$ small quasi-squares.
Now we will decorate  the tiles of $P'$ in some way.

\subsection{The general plan for constructing the decoration of $P$}

To prove Theorem~\ref{th22}(a), it suffices to prove the following 
\begin{prop}\label{th5} There exists a side decoration
of tiles of $P'$ such that a tiling $\mathcal T$ with tiles of the original set $P$ is a $\tau$-substitution tiling 
if and only if $\mathcal T=\Delta\pi\mathcal T'$ for some proper tiling 
$\mathcal T'$ with  decorated tiles (the local rule stipulates that tiles can be connected only side-to-side so that
colors of shared sides match).
\end{prop}
Why is this sufficient? Because we can apply additional colors (we will call them \emph{glues}) to the sides of the colored tiles participating
in each of the added tiles $\alpha$ so as to force
tiles with these colors in a proper tiling to assemble
into the tile $\alpha$ (for each of the added tiles $\alpha$, these will be unique colors).
The resulting set of prototiles $\tilde P$ is the sought one:
\begin{lemma}\label{claim-prop1}
A tiling $\mathcal T$ with tiles from $P$ is a $\tau$-substitution tiling iff $\mathcal T=\pi\tilde{\mathcal T}$
for some proper tiling $\tilde{\mathcal T}$ with tiles from $\tilde P$.
(The proof of this almost obvious lemma is in the Appendix.)
\end{lemma}

To prove Proposition~\ref{th5}, we first define a substitution
$\sigma'$ on $P'$.

\subsection{Substitution $\sigma'$ on $P'$}

Consider first the following substitution $\sigma$. When applied
to the tile $\alpha\in P$, it yields $\tau^k\alpha$, and then combines some tiles in the macrotile
$\tau^k\alpha$ into small and large quasi-squares as described above, that is,
$\sigma\alpha=\Delta^{-1}\tau^k\Delta\alpha$.
Now we extend  substitution $\sigma$ to prototiles from $P'\setminus P$.
The substitution $\sigma$ is applied to a quasi-square $\alpha\in P'\setminus P$ as follows:
We disperse this quasi-square into its original tiles, then apply $\sigma$ to each of them, and then simply combine the resulting 
$\tau^k$-macrotiles into a single tiling:
$\sigma\alpha=\bigcup_{\beta\in\Delta\alpha}\sigma\beta$.

Now, in each $\sigma$-macrotile, we choose a \emph{central tile}.
For $\alpha\in P$, the central tile in $\sigma\alpha$
is defined as the large quasi-square in $\sigma\alpha$.
For $\alpha\in P'\setminus P$, the central tile in $\sigma\alpha$
is defined as the large quasi-square in $\sigma\beta$, where $\beta$ is the upper-left (say) tile
of $\Delta\alpha$.

Further, we will need  that for every large quasi-square
$\alpha\in P'$, the central tile in $\sigma\alpha$ coincide with $\alpha$.
We can achieve this property using the following lemma.

\begin{lemma}\label{l23}
Let $f$ be any function from a finite set $S$ to itself.
Then for some integer $l>0$ we have $f^{2l}=f^l$.
\end{lemma}
\begin{proof}
Indeed, the number of functions mapping $S$ to $S$ is finite. Therefore 
for some $n$ and $l>0$ we have $f^n=f^{n+l}$.
It follows that $f^{n+i}=f^{n+l+i}$ for all $i\ge 0$.
If $l\ge n$, then we set $i=l-n$ and obtain $f^{l}=f^{2l}$.
It remains to note that $l$ can be made arbitrarily large,
since the equality $f^n=f^{n+l}$ implies the equalities $f^n=f^{n+l}=f^{n+2l}=f^{n+3l}=\dots$.
\end{proof}

Consider the function $\cen:P'\to P'$, \label{centr} which maps every prototile $\alpha\in P'$ to the central prototile in
$\sigma\alpha$. By Lemma~\ref{l23},
there exists $l>0$ such that $\cen^{2l}=\cen^{l}$. Let $m=kl$ and $\sigma'=\sigma^l$.
In $\sigma'$-macrotiles, we choose central
tiles recursively --- as central tiles in the images of central tiles, and so on.

Now, in each $\sigma'$-macrotile, we scatter all large quasi-squares that are not central tiles back onto the tiles of $P$
and also scatter small quasi-squares that do not border central tiles.
Now we further remove from $P'$ all unnecessary large quasi-squares; that is, we keep only those large quasi-squares that belong to the image set $\im\cen^l$.
Since in any macrotile the central quasi-square belongs to $\im\cen^l$, no central quasi-square will be removed from any macrotile.

Since
$\cen^{2l}=\cen^{l}$, for any large quasi-square $\alpha\in P'$
the central tile in $\sigma'$-macrotile $\sigma'\alpha$ coincides with $\alpha$.
Indeed, let  $\alpha\in P'$ be a large quasi-square. By the restriction, $\alpha\in \im \cen^l$,
so $\alpha=\cen^l(\beta)$ for some $\beta$ (in old $P'$). Then 
$\cen^l(\alpha)=\cen^{2l}(\beta)=\cen^{l}(\beta)=\alpha$.

So, we have constructed a substitution $\sigma'$ on the set $P'$ that is consistent with some power $\tau^m$
of the original substitution ($\tau^m\Delta\alpha=\Delta\sigma'\alpha$). Furthermore,
each $\sigma'$-macrotile $\sigma'\alpha$ contains a single large quasi-square,
and it coincides with $\alpha$ if $\alpha$ is itself a large quasi-square. 

Next, we will need
the so-called \emph{labeling} of $\sigma'$-macrotiles.

\subsection{Labeling of $\sigma'$-macrotiles}

\subsubsection{Marked sides of prototiles from $P'$}\label{smarked}

Let a $\sigma'$-macrotile $\sigma'\alpha$ be given.
We
\emph{mark} some sides in $\sigma'\alpha$.
In each large quasi-square $L$, for each adjacent small quasi-square $S$, we mark
the middle side among all sides shared by $S$ and $L$  (the total number of such sides is  $4N$).
Also we mark all sides that do not belong to 
large quasi-squares (see Fig.~\ref{f418}). Whether a side of a tile is marked depends only
on the form of the tile and does not depend on the macrotile or the location within the macrotile.
Indeed, this could be violated only for quasi-squares, and quasi-squares occur
only near the centers of macrotiles inside macrotiles $\tau^m\alpha$ for original tiles $\alpha$. 

The number $N$ should  be chosen so
that the number
of marked sides of each large quasi-square is more than three times
$$
Q=\text{the maximal number
of sides of the original prototiles and of small quasi-squares},
$$
that is, $4N>3Q$. The weaker inequality $4N\ge Q$ is all we need in order to choose the prints
below; the stronger one will be used once, in the proof of Proposition~\ref{th2}. Note that both
are easy to satisfy, since $Q$ depends only on the original prototiles and on the size $R$ of the
grid square, and $N$ is chosen after $R$ (see the proof of Lemma~\ref{l8} below).

\subsubsection{Ports and paths in macrotiles}

For each marked side $a$ of prototile $\alpha\in P'$, we choose some marked side $b$ of the central tile
in the macrotile $\sigma'\alpha$, called the \emph{print} of $a$.
The prints must be chosen so that when going around the border of the tile $\alpha$, the sides go in the same order as their prints when going around the border of the central tile in the macrotile $\sigma'\alpha$.
In particular, the print map is injective: distinct marked sides of $\alpha$ have distinct prints.
This is possible, since the number of marked sides of the central tile is at least $Q$.
 Recall that, 
if the prototile $\alpha$ is a large quasi-square, then the central tile in $\sigma'\alpha$ coincides with $\alpha$. In this case, we choose $a$ itself as the print of $a$.


Let a macroside of the macrotile $\sigma'\alpha$ for $\alpha\in P'$ contain $L$ sides.
Consider the sides on it
numbered $\lceil L/5\rceil,\lceil 2L/5\rceil,\lceil 3L/5\rceil,\lceil 4L/5\rceil, L$, counting from
left to right, and call them \emph{ports} (the number $L$ need not be divisible by~5; any fixed
rounding convention will do, and $L$ is large, so these five numbers are distinct).
Obviously, if two macrotiles are adjacent macroside-to-macroside and
side-to-side, then each port is adjacent to a port with the same number.

\begin{definition}
A \emph{path} in a $\sigma'$-macrotile is a sequence of sides $\{a_1,\dots,a_k\}$
of its tiles such that sides $a_i,a_{i+1}$ belong to the same tile for all $i$.
A path \emph{passes through a side} if this side belongs to this path.
\end{definition}
\begin{figure}
\begin{center}
\includegraphics[scale=.8]{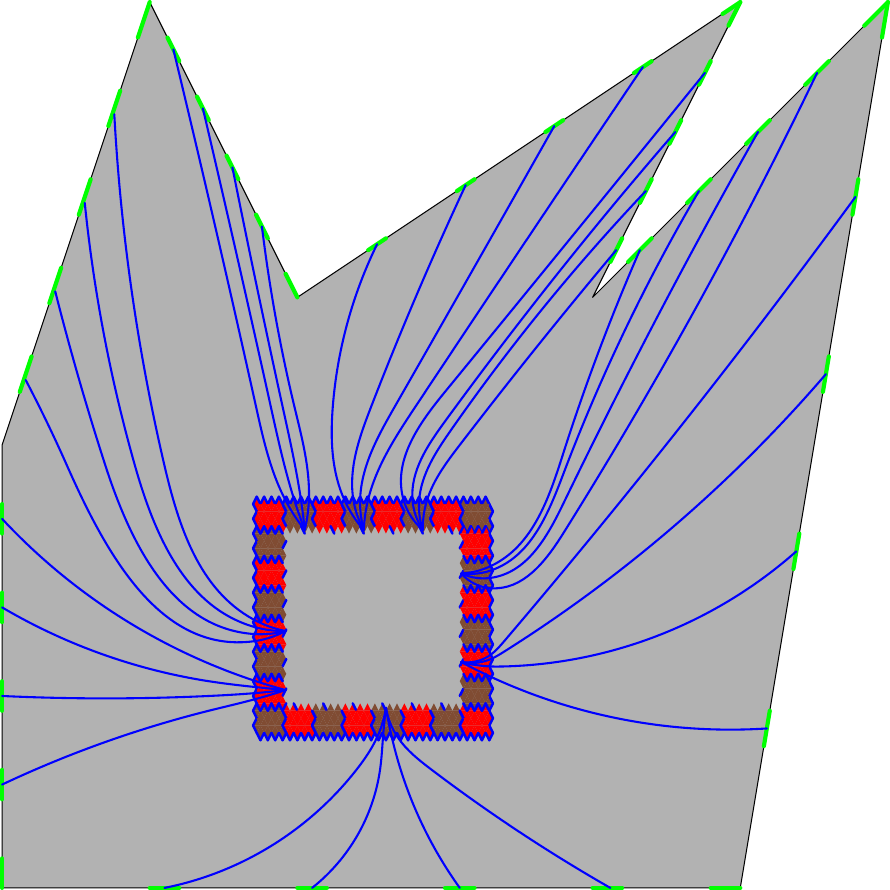}
\end{center}
\caption{Ports and paths in a macrotile. The ports are colored green.
Paths connecting the ports to the prints of the sides of the parent  prototile are drawn in blue.
}\label{f419}
\end{figure}
For each marked side $a$ of $\alpha$ and for each $i\le 4$ we choose a path
that starts at the $i$th port on the macroside $\sigma' a$, ends in 
the print of $a$ and contains  marked sides only. See
Fig.~\ref{f419}, 
where these paths
are drawn in blue. The paths are denoted by $P_1(a), P_2(a), P_3(a), P_4(a)$.

Note that there is no genuine path $P_5(a)$ made of several sides; this is because there will be no need to reconcile the information in the fifth port with the information in the central tile. For the same reason, for a non-marked side $a$ there are no genuine paths at all. Nevertheless, to simplify future definitions, it is convenient to let every side $a$ have five paths $P_1(a), P_2(a), P_3(a), P_4(a), P_5(a)$. To this end, we \emph{define} $P_5(a)$ for a marked side $a$ to consist of a single side --- the fifth port; and we \emph{define} $P_i(a)$ for a non-marked side $a$ and for $i\le 5$ to consist of a single side --- the $i$th port.

We need that
each side of each tile, with the exception
of the sides of the central tile, belong to at most
one path. 
We also require that a path contain no side of the central tile except its last side, that is,
except the print at which it ends. Since the marked sides of the central tile are exactly the
sides it shares with the adjacent small quasi-squares, this means that the side of the path
$P_i(a)$ that precedes the print of $a$ belongs to the small quasi-square sharing that print with
the central tile; that side is a non-special side of that small quasi-square, because a small
quasi-square has exactly one special side.
If the power $\tau^k$ is chosen large enough, then it is possible
to choose such paths.

Now we can begin coloring the prototiles in $P'$.

\subsection{Decoration of the set $P'$}\label{s43}

\subsubsection{Types and indices}
We will call tiles in the union of the $\sigma'$-macrotiles $\sigma'\alpha, \alpha\in P'$, the \emph{$\sigma'$-types} (we assume that these macrotiles are disjoint). We denote the set of all types by $\types$.
\begin{center}
\includegraphics[scale=1.]{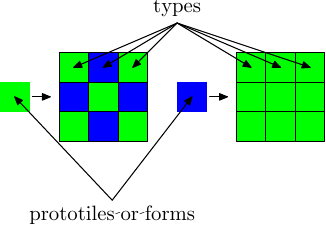}
\end{center}
We will distinguish
four kinds of sides of $\sigma'$-types:
\begin{enumerate}
\item Sides lying on the boundary of the macrotile are called \emph{outer sides}. All other
  sides are called \emph{inner}. 
\item \emph{Interior sides} are inner sides
whose both ends do not lie on the boundary of the macrotile.
\item  \emph{Special} sides are marked sides of central types and also marked sides of
small quasi-squares shared with the large quasi-square.
\item Inner sides that have one or two ends on the boundary 
  are called \emph{boundary sides}.
\end{enumerate}
Note that if two types in a macrotile share a side,  then this side is of the same kind 
in those types.

The color of each side of a tile will be a 5-tuple 
of  so-called \emph{indices}. Now we introduce those indices.
Let $s$ be a type in $\sigma'\alpha$.

\begin{itemize}
\item The \emph{identity} index of a side $b$ of type $s$ is defined as follows:
\begin{itemize}
\item \label{redind} If $b$ is an inner side of type $s$,
then let $r$ denote the type
of the tile from $\sigma'\alpha$ that lies in $\sigma'\alpha$ on the other side of $b$.
Then the identity index
is equal to the ordered pair $\pair{s,r}$ or $\pair{r,s}$
depending on whether $s$ is to the left or right of $b$.
If side $b$ is horizontal, then the top tile is considered the left one.

\item
If $b$ is an outer side of $s$, then this index is 0.

The 
identity index of a side $b$ of type $s$
is denoted by $I_s(b)$.
\end{itemize}
Identity indices will cause the tiles of a proper tiling
to be assembled into $\sigma'$-macrotiles.
\item The alignment index of a side $b$ of type $s$ is defined as follows:
If $b$ is an outer side of type $s$,
then the alignment index is defined as the number of the side in left-to-right order
along the macroside $\sigma' a$ that $b$ belongs to.
Otherwise, it is zero. Alignment indices force
$\sigma'$-macrotiles to be connected macroside-to-macroside.
The alignment index of a side $b$ of type $s$
is denoted by $A_s(b)$.
\item
The  \emph{crown} index of a side $b$ of type $s$ is a list
of at most two
$\tau$-allowed crowns, defined as follows.
If side $b$ has no vertices
on the boundary of the macrotile, then the crown index is zero (the list is empty).
If side $b$ has one vertex $V$ on the boundary, then
the list consists of a single crown that includes $s$. If $s$ is a quasi-square, then this means that the
crown must include all original tiles from $\Delta s$ that include the vertex $V$. 
If side $b$ has two vertices $V,V'$ lying on the boundary, then
the list consists of two crowns $C,C'$, first for the left vertex, then for the right one, 
such that both $C,C'$ include $s$. 

The set of possible crown indices of side $b$ of type $s$
is denoted by $C_s(b)$.
\end{itemize}

Now we construct some intermediate set of colored tiles $P''$,
from which we then remove some tiles to obtain the final set $P'''$.

\subsubsection{Colored set $P''$}
\newcommand{\ind}{\textit{ind}}

On each side of each tile in the set $P''$, five indices will be written
numbered $1,\dots,5$ and ranging through the set
\begin{align*}
U=\types\cup(\types\times\types)\cup\{0,1,\dots,K\}\cup\crown\cup(\crown\times\crown)
\cup\{1,\dots,K\}^m\cup\{\undef\},
\end{align*}
where $\types$ denotes the set of all types, and $K$ is the maximum number of sides
in the macrosides of $\sigma'$-macrotiles. 
One more index will be written in the middle of the tile,
we call it \emph{the sixth} index.
This index does not affect the way the tiles are connected, so we will remove it later. On all tiles in set $P'''$, the sixth index will indicate the type of the tile.
The $i$th index on side $a$ of tile $A$ will be denoted by $\ind_A(i,a)$.

The set of tiles $P''$, by definition, consists of all tiles in 
$P'$, each of whose sides is colored as a tuple of five elements of $U$,
and one more element of $U$ is written in the middle of the tile.

Now we define the substitution $\sigma''$ on decorated tiles.
Ignoring the colors, i.e., the indices, this will be the same substitution
$\sigma'$.

\subsubsection{Substitution $\sigma''$ on tiles of $P''$}\label{s263}
\newcommand{\cld}{\text{Ch}}

This substitution is non-deterministic, since some 
crown indices can be chosen in multiple ways.
That is, for a tile $A\in P''$ of the form $\alpha$,  $\sigma'' A$ consists
of several 
tilings $\mathcal M$ with $\pi\mathcal M=\sigma'\alpha$.

Let $A\in P''$ be a tile of form $\alpha$. Consider the macrotile $\sigma'\alpha$.
Recall that we call tiles from the macrotile $\sigma'\alpha$ types.
Let $\mathcal M$ be a tiling with  tiles from $P''$ such that $\pi\mathcal M=\sigma'\alpha$. By $\mathcal M_s$
we denote the tile from $\mathcal M$ obtained by choosing indices
for type $s$. 

Tilings $\mathcal M\in \sigma'' A$ are defined in such a way that
$\mathcal M_s$ stores $s$, the identity and crown indices of some sides of type $s$
(this information is independent of the indices of tile $A$), and also
some information about the indices of tile $A$. More precisely, a tiling $\mathcal M$ belongs to $\sigma'' A$ if the following conditions are met:
\begin{itemize}
\item
The sixth index of $\mathcal M_s$ is $s$. For this reason, we call a tile $A\in P''$
with sixth index $s$ a \emph{tile of type $s$}.

\item If $b$ is a special side of type $s$ and $b$ is the print of a side $a$ of the tile $A$,
then the color of side $b$ of tile $\mathcal M_s$
is
$$\pair{\ind_A(1,a),\ind_A(2,a), \ind_A(3,a),\ind_A(4,a),I_s(b)}.$$
Being a print is a property of the side $b$ of the macrotile $\sigma'\alpha$, and not of the pair
(side, type): every print is a side shared by the central type and an adjacent small quasi-square,
and it is special in both of them. This clause therefore applies to \emph{both} types sharing $b$,
and it prescribes to both of them the same color, since the identity index $I_s(b)$ of an inner
side is the same ordered pair for the two types sharing that side. This symmetry is essential:
were the borrowed 4-tuple required from the central type only, no macrotile in $\sigma''A$
would be proper.

\item Otherwise
\begin{enumerate}
\item The first index on side $b$ is equal to the
identity index of side $b$ of type $s$, that is, to $\ind_{\mathcal M_s}(1,b)=I_s(b)$,

\item The second index on side $b$ is equal to some
crown index of side
$b$ of type $s$, that is, $\ind_{\mathcal M_s}(2,b)\in C_s(b)$,
\item The third index on side $b$
is equal to the $i$th index on side $a$ of tile $A$, that is,
$\ind_{\mathcal M_s}(3,b)=\ind_A(i,a)$ if $b\in P_i(a)$ (recall that there is at most one such path) 
and $\ind_{\mathcal M_s}(3,b)=0$ otherwise. In particular, if  $b$ is the $i$th port on macroside $\sigma' a$,
then $\ind_{\mathcal M_s}(3,b)=\ind_A(i,a)$ (recall that  $i$th port on macroside $\sigma' a$ always belongs to $P_i(a)$).
\item
The fourth index on side $b$
is equal to $A_s(b)$ if
$b$ is an outer side of type $s$,
is equal to the sixth index of the tile
$A$ if $b$ is a boundary side of type $s$, and is zero otherwise.
For this reason, the fourth index of boundary
sides will be called \emph{the parent index}.
\item
The fifth index on side $b$
is zero. \end{enumerate}
Finally, there is one more constraint on the choice of the second index of outer and boundary sides:
\item \textit{Crown Consistency Property:} if sides $b,b'$ of a type $s$ share a vertex $V$
which lies on the boundary of $\sigma'\alpha$,
then the crowns for $V$ in $\ind_{\mathcal M_s}(2,b)$ and $\ind_{\mathcal M_s}(2,b')$ must match.

\end{itemize}
Note that not all macrotiles $\mathcal M\in\sigma'' A$ are proper tilings.
This is because the values of the second index 
on a common side of two different tiles may differ.

\begin{definition}
We call a tile $B\in P''$ a \emph{child of type $s$} of a tile $A\in P''$ if $B=\mathcal M_s$ for some macrotile
$\mathcal M\in \sigma''A$. In this case, we also call $A$ the \emph{parent} of $B$.
The set of all children of type $s$ of a tile $A$ will be denoted by $\cld_s(A)$.
\end{definition}

\begin{definition}[Local indices and normal tiles]
Let $s$ be a type.  
\emph{Local
indices of type $s$} are the following
\begin{itemize}
\item if $a$ is a non-marked side, then all its indices are local,
\item if $a$ is a special 
side, then its fifth index is  local,
\item otherwise the first and fifth indices are local.
\end{itemize}
The idea behind this definition is that local indices on a side $a$ of a tile $A$ are those that 
can be found from the type  of $A$ and from $a$. More exactly,
if a tile's local indices and its type  (that is, its sixth index) satisfy
the above requirements, then we call the tile \emph{normal}.
\end{definition}

\begin{remark}\label{rem-parent}
A tile may have several parents or no parents at all. In the former case, all its parents
are of the same form. If $B$ is a boundary tile, then, moreover, all its parents
are of the same type. All tiles that have a parent are normal.
\end{remark}

\subsection{The Set of Tiles $P'''$}\label{slt}

Now we will discard some tiles from $P''$, imposing some
restrictions on the decoration
of the prototiles. Tiles that satisfy
those restrictions will be called
\emph{legal}. They will form the set of tiles $P'''$.

For each specific initial substitution $\tau$, we can explicitly define the legal tiles,
but since we want our construction to be general, we will use
a different approach.
Obviously, every legal tile must have a parent.
Indeed, a proper tiling with legal tiles
must have a $\sigma''$-composition.
Therefore, we need to remove all  tiles that have no parents.
After that, we must remove
tiles all of whose parents have been removed. And so on. Each new removal may increase the number
of tiles all of whose parents have been removed. More or less, we will call a decorated tile legal if it has a parent, which in turn
has a parent, and so on, infinitely many times.

But unfortunately, this is not enough. For example, let $s$ be any central type, and $t$ 
the central type in the macrotile $\sigma' s$. Then the macrotiles
$\sigma' s$ and $\sigma' t$ coincide, since $s$ and $t$ are of the same form. Therefore, $t$ is a central type in the macrotile $\sigma' t$.
We will call such types \emph{cyclic}.
Every normal
tile of a cyclic type $t$ 
is its own parent.
Therefore, such a tile will never be removed. However, there are too many such
tiles, and they lead to parasitic tilings.
Therefore, we introduce another requirement.

\begin{definition}[Borrowed indices]
Assume that $B$ is a child of $A$. If the side $b$ of $B$ is special and is 
a print of some side, then 
the first four indices of $B$ on $b$ 
are borrowed from the parent tile $A$.
For this reason, these indices are called \emph{borrowed}.
If $b$ is not special but belongs to a path, then its 
third index is also called \emph{borrowed}. 
Finally, the fourth index
on boundary sides is also  \emph{borrowed}.
\end{definition}

\begin{definition}
We call a decorated tile $A\in P''$ \emph{legal} if
there exists an infinite sequence of decorated tiles
$A_0=A,A_1,A_2,\dots\in P''$ with the following properties:
\begin{itemize}
\item $A_{i+1}$ is a parent of $A_{i}$ for all $i$,
\item
Consider some borrowed index $I_i$
in $A_{i}$. It is a copy of an index $I_{i+1}$ from $A_{i+1}$, call $I_{i+1}$ the \emph{source} of $I_i$. This index may again be borrowed, in which case we consider its source $I_{i+2}$ in $A_{i+2}$, and so on.
If the sequence $I_{i}, I_{i+1}, I_{i+2}, \dots$ ends with an unborrowed index,
then we call the latter the \emph{origin} of all indices in the sequence.
Otherwise, we say that $I_i$ \emph{has no origin}.
It is required that for all $i$, all borrowed indices in $A_{i}$ without an
origin are 
zeros.
\end{itemize}
\end{definition}
It is useful to realize where borrowed indices without an origin come from. Parent indices always have an origin.
Therefore, a borrowed index without an origin on a side $a$ can occur only if $a$ belongs to some path. If side $a$ is not special, then this index must be the third one; otherwise it must be the 1st, 2nd, 3rd or 4th one.

When talking about legal tiles, we will freely use the expressions
``parent index'', ``crown index'', ``alignment index''.
Each time, it will be clear which index is meant.

\subsubsection{Properties of Legal Tiles}
To get used to legal tiles,
let us formulate some of their simple properties:

\renewcommand{\labelenumi}{S\theenumi:}
\begin{enumerate}
\item Every legal tile has a parent and hence is  normal --- its parent is 
the second term $A_1$ of the
ancestor sequence that confirms legality.
\label{s1}
\item Every child of every legal tile is also legal.\label{s3}
\item The parent indices on different sides of a legal tile
match. \label{s2}
\item
If a legal tile has two sides on the path $P_i(a)$
then the copies of  $\ind_A(i,a)$
on those sides 
match.\label{s6}
\item
If two sides $a,b$ of the same tile share a common vertex $V$, then the crowns for $V$
in the crown indices of these sides match.\label{s7}
\end{enumerate}

\subsection{Proof of the correctness of the construction}
Now we have to prove that  \emph{a tiling $\mathcal T$ with tiles of the original set $P$ is 
a $\tau$-substitution tiling if and only if $\mathcal T=\Delta\pi\mathcal T'$ for some proper tiling
$\mathcal T'$ with legal tiles.}

First, we prove the easy direction, that it is possible
to properly color any substitution tiling.

\begin{prop}\label{th1}
For any $\tau$-substitution  tiling  $\mathcal T$ with tiles from $P$ there is 
a proper tiling $\mathcal T'$ with legal tiles such that $\mathcal T=\Delta\pi\mathcal T'$. 
\end{prop}
\begin{proof}
  Fix a $\tau$-substitution  tiling $\mathcal T$ and an infinite sequence $\mathcal T_0=\mathcal T, \mathcal T_1,\mathcal T_2,\dots$ of substitution tilings  in which each tiling is a $\tau$-composition of the previous one. 
Such a sequence exists by Lemma~\ref{th11}.
  Consider its subsequence $\mathcal T_0, \mathcal T_m,\mathcal T_{2m},\dots$.
In this subsequence each tiling is a $\tau^m$-decomposition of the next one.

In each tiling  $\mathcal T_{im}$ group the tiles into $\tau^m$-macrotiles according
to $\mathcal T_{(i+1)m}$. 
Then group some tiles to obtain large and small quasi-squares as described above. 
In this way we get a tiling  
$\mathcal T'_{im}$ with tiles from $P'$. 
Moreover, by construction we have $\mathcal T'_{im}=\sigma' \mathcal T'_{(i+1)m}$.

Now we have to decorate $\mathcal T'_{im}$. Since $\mathcal T_{(i+1)m}$ is side-to-side,
  in $\mathcal T_{im}$ every $\sigma'$-macroside is adjacent to a $\sigma'$-macroside. Hence 
each port is adjacent to a port with the same number in $\mathcal T'_{im}$.
  
  We want to properly color all tilings of our sequence so that for the resulting sequence
  $\mathcal T''_0, \mathcal T''_{m},\mathcal T''_{2m},\dots$, each tiling is a $\sigma''$-composition of the previous one. 
  To do this, we color each macrotile $\mathcal {M}$ from
the tiling $\mathcal T''_{im}$ as described in Section~\ref{s263}. 
We choose the components of 
crown indices as actual crowns in $\pi\Delta\mathcal T_{im}$. The construction ensures that
all non-borrowed indices on shared sides match and all crowns in crown indices are $\tau$-allowed. 

  The borrowed indices of the tiles from $\mathcal {M}$ are yet undefined since we have not yet completely decorated the parent $D$ of $\mathcal {M}$, which belongs to $\mathcal T_{(i+1)m}$.
  We set the borrowed indices that have an origin to their origin, and set the borrowed indices without an origin to zero.
  
  This decoration has the following properties.
 
  (a) \emph{All tilings of the chain $\mathcal T''_0, \mathcal T''_m,\mathcal T''_{2m},\dots$ are proper.}
  
Indeed, non-borrowed indices on shared sides match by construction. Let us prove the same thing for borrowed indices.
Consider a pair of  borrowed indices $\ind_B(k,b)$ and  $\ind_C(k,b)$
such that $b$ is a shared side of tiles $B,C\in  \mathcal T''_{im}$.
If $b$ is an inner side, then by construction $\ind_B(k,b)$ and  $\ind_C(k,b)$ copy the same index
of the parent tile and hence match.
Otherwise $b$ is an outer side of a macrotile from $\mathcal T''_{im}$.
Then $B,C$ belong to different macrotiles, thus they copy different indices.
However, if  $\ind_B(k,b)$ has an origin $\ind_D(l,d)$ in a tile $D\in \mathcal T''_{(i+j)m}$,
then $\ind_C(k,b)$ also has an origin and that origin is either the same 
index $\ind_D(l,d)$, or it is 
$\ind_E(l,d)$ where
$E$ is the tile from $\mathcal T''_{(i+j)m}$ that shares the side $d$ with $D$.
This is proven by induction on $j$.

In the first case $\ind_B(k,b)=\ind_D(l,d)=\ind_C(k,b)$. In the second case
by construction $\ind_D(l,d)=\ind_E(l,d)$ and hence $\ind_B(k,b)=\ind_C(k,b)$.
The same argument applies when  $\ind_C(k,b)$ has an origin. Otherwise 
both  indices are borrowed and have no origin. Hence they  are equal to zero
and thus match.
 
  (b) \emph{All the resulting decorated tiles are legal.} Indeed, it follows from the construction that after decoration,
  each macrotile is a $\sigma''$-decomposition of its parent.\footnote{The Crown Consistency Property
  is met, since crown indices include actual crowns.} Consider an arbitrary decorated tile $A$ from any of the tilings of the chain
  $\mathcal T'_0, \mathcal T'_{m},\mathcal T'_{2m},\dots$.
  Consider its ancestors
  $$
  \dots\to A_{3m}\to A_{2m}\to A_m\to A_0=A
  $$
  By construction, all indices without an origin in tilings from this sequence are zero. Therefore, the tile $A$ is legal.
\end{proof}

\begin{remark}
To understand what follows, it is helpful to picture the coloring of the paths
constructed in the proof of this proposition. For any inner side $a_{im}$ in the tiling $\mathcal T''_{im}$,
its $j$th index (for $j\le4$) turns, in the tiling $\mathcal T''_{(i-1)m}$,
into two paths $P_j(a_{im})$ lying in two neighbouring macrotiles. These two paths connect
the prints of $a_{im}$ in the central tiles of these macrotiles to the $j$th port on the macroside $\sigma' a_{im}$.
On all sides of these paths, except the prints of $a_{im}$, the third index equals the $j$th index of the side $a_{im}$.
It is convenient to think of them as a single path connecting the two special sides
(the prints of $a$) of the centers of the neighbouring macrotiles.

Then, in the tilings $\mathcal T''_{(i-2)m},\dots,\mathcal T''_{0}$, this combined path turns
into several consecutive paths $P_3$, the third index on the sides of which equals the $j$th index of the side $a_{im}$.
Again it is convenient to picture this as a single path connecting two special sides of tiles
that can be very far from each other.
Thus this index is stored on the sides of the combined path in each of the tilings
$\mathcal T''_{(i-1)m},\dots,\mathcal T''_{0}$.

But in the tiling $\mathcal T''$ there may also be consecutive paths of the same kind that store
no information in their third index --- on them the third index equals $0$. This happens when
the third index on the side $a$ has no origin. This occurs when the side $a$ belongs to the path $P_3(a_1)$
connecting the third port on the macroside $\sigma'(a_1)$ to the print of $a_1$, while the side $a_1$
belongs to the path $P_3(a_2)$ connecting the third port on the macroside $\sigma'(a_2)$ to the print of $a_2$,
and so on, infinitely many times. 
\end{remark}

Let us now prove that any proper tiling with legal tiles has a $\sigma''$-composition
that is proper and consists of legal tiles only.

\begin{prop} \label{th2}
  Every proper tiling $\mathcal T$ with legal tiles has a 
proper $\sigma''$-composition $\mathcal T'$ consisting of legal tiles.
 \end{prop}
\begin{proof}
Recall that a central type $t$ is called \emph{cyclic} if $t$ is the central type in the macrotile $\sigma' t$, see Fig.~\ref{f34}.
\begin{figure}
\begin{center}
\includegraphics[scale=1]{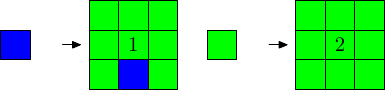}\hskip 2.5cm \includegraphics[scale= 1]{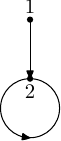}
\end{center}
\caption{The central type 1 is acyclic, and the central type 2 is cyclic.}\label{f34}
\end{figure}
The key lemma in the proof is the following 
\begin{lemma}[Composition Lemma]\label{l9}
Let $\mathcal M$ be a finite proper tiling 
that is equal to a $\sigma'$-macrotile 
$\sigma' \alpha$  provided we ignore all indices except the sixth one (the one written in the
middle of a tile, which specifies its type).
Assume that all tiles in $\mathcal M$ are legal, except possibly for the  central tile $A$. 
    Then the following hold: 
    \begin{enumerate}
    \item[(a)] All boundary sides of  $\mathcal M$ have the same parent index $t$, which is of the form $\alpha$.
  \item[(b)] There is a normal tile $D$ of type $t$ with $\sigma''D\ni\mathcal M$.
      
    \item[(c)] If the central tile $A$ is legal and  $t$ is a non-central or cyclic type, then $D$ is legal as well.
    \end{enumerate}
  \end{lemma}
  \begin{proof}
    (a) Due to property~S\ref{s2} and since $\mathcal M$ is proper, all the boundary sides of  $\mathcal M$ have the same parent index $t$. To prove that $t$ is of the form $\alpha$, 
it suffices to find a legal type-$t$ tile of the form $\alpha$.
   Such a tile is any legal parent $C$ of any border tile $B$ in $\mathcal M$,
   see Fig.~\ref{f8}.
    \begin{figure}
      \begin{center}
        \includegraphics[scale=.9]{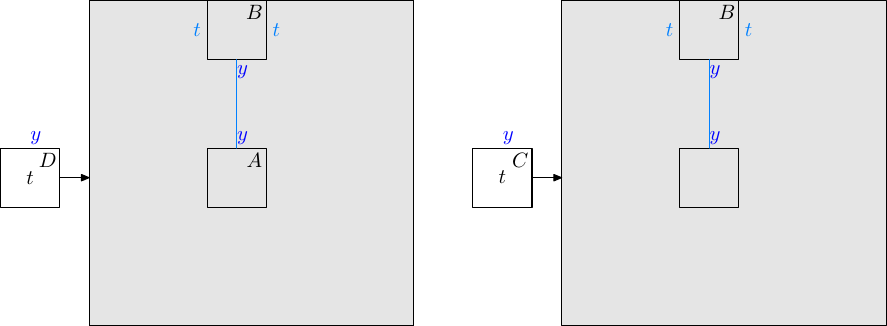}
      \end{center}
      \caption{On the left is a proper tiling $\mathcal M\in \sigma'' D$ with $\pi(\mathcal M)=\sigma'\alpha$ 
and with the  central tile $A$. 
        On the right is the legal parent $C$ of $B$ and a decomposition of $C$. 
        The type of $C$
        and the parent indices in its decomposition and in the tiling $\mathcal M$ are  denoted
        by $t$.
}\label{f8}
    \end{figure}
   The form of $C$ is $\alpha$, since $C$ has a child in $\mathcal M$ and $\pi(\mathcal M)=\sigma'\alpha$. And the type of $C$ is $t$, since so is
 the parent index of $B$.

(b) Let $D$ be the following tile, whose indices are determined
by $t$ and third  indices on ports of  $\mathcal M$:
\begin{itemize}
\item the sixth index of $D$ is $t$,
\item the $i$th index of the side $d$ of $D$ is equal to the third index on the $i$th 
port on the macroside $\sigma'd$ of $\mathcal M$.
\end{itemize}

Let us show that this tile $D$ is normal. Let $\ind_D(i,d)$ be a local
index of $D$. 
By construction, 
$$
\ind_D(i,d)=\ind_B(3,b)
$$
where $b$ is the $i$th port
on the macroside $\sigma'd$ of the macrotile $\mathcal M$ and 
$B$ the tile from $\mathcal M$ containing this port (see Fig.~\ref{f8}). 
By assumption,
tile $B$ is legal, and therefore has some legal parent $C$.
The type of $C$ is $t$, since that is the parent index of $B$. Indeed, the tile $B$ contains a
port and hence has an outer side; not all sides of $B$ are outer, so $B$ has a boundary side as
well. By item (a) the parent index on that side equals $t$, and by the requirement on the fourth
index of boundary sides it equals the sixth index of $C$, that is, the type of $C$.
Furthermore,
$$
\ind_B(3,b)=\ind_C(i,d)
$$
by the definition of the substitution $\sigma''$.
Therefore, $\ind_D(i,d)=\ind_C(i,d)$,
and therefore $\ind_D(i,d)$ takes the value allowed for normal tiles of type $t$.
Note that in this argument it matters only that the value of a local index for a normal 
tile is a function of $t$ and $d$.

Let us show
that $\mathcal M$ is a $\sigma''$-decomposition of $D$.
First, note that $\mathcal M$ satisfies the Crown Consistency Property,
since all the tiles that have vertices on the border of $\mathcal M$ are  legal and 
thus satisfy Property~S\ref{s7}. Therefore,
it suffices to prove that
for all types $s$ in $\sigma'\alpha$, the tile $\mathcal M_s$
has indices satisfying the requirements for indices of tiles
$\mathcal M_s$ for $\mathcal M\in\sigma''D$.

First, assume that tiles of type $s$ do not have borrowed indices; in particular, type $s$ is neither central nor boundary.
Then, each tile of the form $\alpha$
has only one child of type $s$, and for different tiles of the form $\alpha$,
these children coincide.
Since the tile $\mathcal M_s$ is legal, it has a parent and that parent must be of the form $\alpha$.
Therefore, $\mathcal M_s$ is a child of $D$.

Now suppose that type $s$ has borrowed indices but is not central. Let $\ind_{\mathcal M_s}(j,b)$ be any index of the tile $\mathcal M_s$.
We need to prove that this index satisfies the requirements for indices
of tiles from macrotiles in $\sigma'' D$.
By assumption, the tile $\mathcal M_s$ is legal, and therefore
has a legal parent $C$. Therefore, $\ind_{\mathcal M_s}(j,b)$
satisfies the requirements for macrotiles in $\sigma'' C$. If
$\ind_{\mathcal M_s}(j,b)$ is not borrowed, then these requirements for
$\sigma'' C$ and $\sigma'' D$ are the same, so $\ind_{\mathcal M_s}(j,b)$
also satisfies the requirements for macrotiles from $\sigma'' D$.

If
$\ind_{\mathcal M_s}(j,b)$ is borrowed, then the argument is a little more complicated.
Recall that there are three sorts of borrowed indices: the parent index on a boundary side, the
third index on a non-special side belonging to a path, and the first four indices on a special
side that is a print. We consider these three cases in turn.

\emph{The parent index} ($j=4$ and $b$ is a boundary side). Then $\ind_{\mathcal M_s}(j,b)$
must be equal to the type of tile $D$. This is indeed the case
by the construction of tile $D$.

\emph{The third index on a path.} Assume that $b$ is not special and $j=3$, so that
$b\in P_i(d)$ for some $i\le5$ and some side $d$ of tile $D$.
In this case, we need to prove that $\ind_{\mathcal M_s}(3,b)=\ind_D(i,d)$.
Indeed, all sides of the path $P_i(d)$ in the macrotile $\mathcal M$, except the print at which
the path ends, have the same third index. For two sides of the path that belong to the same tile
this follows from Property~S\ref{s6}: both their third indices are copies of one and the same
index of a legal parent of that tile. For two sides of the path that are shared by neighbouring
tiles this follows from the properness of $\mathcal M$.
In particular, $\ind_{\mathcal M_s}(3,b)$ is  equal to the third index of the $i$th port of the macrotile $\mathcal M$
on the macroside
$\sigma' d$. By the construction of the tile $D$, the latter is equal to $\ind_D(i,d)$.

\emph{The borrowed indices on a print.} Assume finally that $b$ is special and is the print of a
marked side $d$ of the form $\alpha$; then $j\le 4$, and what we have to prove is that the color
of $b$ in $\mathcal M_s$ equals
$\pair{\ind_D(1,d),\ind_D(2,d),\ind_D(3,d),\ind_D(4,d),I_s(b)}$, as $\sigma''D$ requires. Since
the type $s$ is not central, $s$ is the small quasi-square that shares the print $b$ with the
central type. The tile $\mathcal M_s$ is legal, hence it has a legal parent $C$, and the form of
$C$ is $\alpha$, because $C$ has a child of type $s$. Consequently $C$ and $D$ have the same
prints and the same paths, and the requirements of $\sigma''C$ tell us that the color of $b$ in
$\mathcal M_s$ is $\pair{\ind_C(1,d),\dots,\ind_C(4,d),I_s(b)}$. So it suffices to show that
$\ind_C(j,d)=\ind_D(j,d)$ for every $j\le4$.

Consider the side $c$ of the path $P_j(d)$ that precedes the print $b$. The path has at least
three sides, since no tile of the macrotile has both an outer side and a side of the central tile.
By the choice of the paths, $c$ is a non-special side of the small quasi-square $s$ itself, so the
requirements of $\sigma''C$ give $\ind_{\mathcal M_s}(3,c)=\ind_C(j,d)$. On the other hand, by the
previous case the third index is the same on all sides of $P_j(d)$ except the print, so
$\ind_{\mathcal M_s}(3,c)$ equals the third index of the $j$th port of $\mathcal M$ on the
macroside $\sigma'd$, which is $\ind_D(j,d)$ by the construction of $D$. Hence
$\ind_C(j,d)=\ind_D(j,d)$, which is what we needed.

   It remains to show that $A$, the central tile in $\mathcal M$, is in  $\cld_s(D)$, where
   $s$ is the central type in $\mathcal M$. 
     We have to prove that on each side $a$ of $A$ all
    indices are as required by the definition of $\cld_s(D)$. Choose any such side $a$.
    On that side, $A$ has the same indices as its neighbour $B$ on side $a$,
    since the tiling $\mathcal{M}$ is proper.
    And $B$  has the indices required for the child of $D$ of its type.
    By definition of $\sigma''$,  the indices of adjacent children on shared interior sides coincide.
    Therefore, the tile $A$ has the required indices on the side $a$. Note that here we use the
    symmetry of the requirement on prints: if $a$ is a print, then $\sigma''D$ prescribes on $a$
    the same borrowed tuple to the central type and to the small quasi-square sharing $a$ with it.

    (c) 
First, consider the case where $t$ is a cyclic type. We
claim that in this case $D$ and $A$ coincide, hence $D$ is legal.
Indeed, by construction, $D$ is of type $t$, thus the central child of $D$,
that is, the tile $A$, is also of type $t$. Both tiles $D,A$ are normal, therefore 
they have the same local indices. Finally, since $A$ is 
the central child of $D$,
its nonlocal indices copy those of $D$.

We now turn to the case where $t$ is a non-central type.

Consider any chain of tiles
$$
A_0=A,A_1,A_2,\dots\in P'',
$$
that proves the legality of $A$.
We claim that at least one tile in this sequence has a non-central type.
Indeed, otherwise, the first index on all marked sides of $A$ would be zero.
But we know this is not so, since $A$ inherits its first index from $D$,
which is not zero on at least one side; recall that $D$ is a normal non-central tile
and hence has an inner marked non-special side (prints are defined for marked sides only, and both
kinds of non-central types --- original tiles and small quasi-squares --- do have such a side).

Now consider the shortest chain
$$
A_0=A,A_1,A_2,\dots,A_n\in P'',
$$
in which each tile is a central child of the next one, and the last tile $A_n$
is of a non-central type and is legal. 
The central child of a central tile has the same form as the tile itself, so their central children
have the same type. Therefore, the types of all the tiles $A_0,\dots,A_{n-2}$ coincide; denote this common type by $s_0$.
Since $A_0$ is the central child of $A_1$ and has type $s_0$ too, the central child of a type-$s_0$ tile again has type $s_0$,
that is, $s_0$ is cyclic. By the cyclic case treated above, a normal tile of a cyclic type coincides with its central child;
hence the tiles $A_0,\dots,A_{n-2}$ all coincide.
Hence we can remove from the chain all the tiles $A_1,\dots,A_{n-2}$, which implies that
$n\in\{1,2\}$.

Let us show that in fact $n=1$. Assume that $n=2$. The tiles $A_1$ and $D$ are both parents
of $A$, hence by Remark~\ref{rem-parent} they have the same form, namely the form $\alpha$ of the
type $t$. On the other hand, if $n=2$, then $A_1$ is the central child of $A_2$, and the central
tile of a macrotile is always a large quasi-square, so the form of $A_1$ is a large quasi-square.
But the form of the non-central type $t$ is a small quasi-square or an original prototile, and
never a large quasi-square --- a contradiction. Thus $n=1$; let $s$ denote the type of the legal
tile $A_1$.

We claim that  $s=t$. To prove the claim, 
denote by $\Phi(A)$ the multiset consisting of all types $r$ such that 
there is a type $u$ such that  $\pair{r,u}$ or $\pair{u,r}$ is a first index of a \emph{marked} side of $A$.
The cardinality $|\Phi(A)|$ is 2 times the number $\marked(A)$ of marked inner sides of $A$ provided
$A$ has no special sides. Indeed, the first index on every non-special marked inner side is the identity
index, which contributes 2 to $|\Phi(A)|$. If $A$ has a special side (which can happen only if $A$ is a small quasi-square),
then $|\Phi(A)|$ may be $2\marked(A)-2$, since the first index on the special side of $A$ can be arbitrary.

We have  $\Phi(A_1)=\Phi(D)$; in fact the first four indices of $A_1$ and of $D$ coincide side by
side. Indeed, $A_1$ and $D$ are parents of one and the same tile $A$, so by
Remark~\ref{rem-parent} they are of the same form $\alpha$, and therefore one and the same print
map $a\mapsto b(a)$ is used for both of them (prints are chosen for a prototile, so they depend
only on the form). All marked sides of the central child $A$ are special, and $b(a)$ is a print, so
the definition of $\sigma''$ gives, for every marked side $a$ of $\alpha$ and every $j\le4$,
$$
\ind_{A_1}(j,a)=\ind_A(j,b(a))=\ind_D(j,a).
$$
In particular, the multisets $\Phi(A_1)$ and $\Phi(D)$ coincide.

Let us stress that the last argument compares $A_1$ and $D$, and \emph{not} $A$ and $D$:
the equality $\Phi(A)=\Phi(D)$ is in general false. The point is that all $4N$ marked sides of the
central child $A$ are special, while the number of prints equals the number of marked sides of
$\alpha$, which is smaller than $4N$ when $\alpha$ is not a large quasi-square; on the special
sides of $A$ that are not prints the first index is the identity index, and it contributes to
$\Phi(A)$ pairs that have nothing to do with $D$.

Thus it suffices to prove the following

\begin{lemma}\label{l8}
Assume that $A,B$ are normal non-central tiles of different types $s,t$. Then $\Phi(A)\ne \Phi(B)$. 
\end{lemma}
\begin{proof}
By way of contradiction assume that $\Phi(A)= \Phi(B)$. 

Without loss of generality we may assume that the number of marked sides of each small quasi-square 
is much larger than the number of sides of any of the original tiles.
Indeed,  let $\Gamma$
denote the largest side length of the prototiles from $P$.
Then the number of marked sides of any small quasi-square is at least roughly $3R/\Gamma$; recall that $R$ denotes
the side length of the grid square. 
Here the factor 3 comes from the fact that 
the side of the small quasi-square that is adjacent to the large quasi-square has only a few marked sides.

Recall that we may choose $R$ arbitrarily large, since
the power $\tau^k$ of the original substitution can be chosen arbitrarily
large. Indeed, it is enough that a grid of size $(N+2)\times (N+2)$ and $4\cdot4\cdot N$ disjoint paths fit inside the
$\tau^k$-macrotiles, where $N$ depends only on the original set and on $R$. Thus
we can first choose $R$, then choose $N$, and only afterwards choose a
sufficiently large $k$.

Hence $\Phi(A)=\Phi(B)$ implies that both $A,B$ are either original tiles having the same number of inner sides,
or small quasi-squares, whose number of marked sides differs by at most 1.

Note also that in both cases the types $s$ and $t$ must lie in one and the same macrotile, since
distinct macrotiles are disjoint. Indeed, if $A$ is an original tile, then it has no special sides,
so every member of $\Phi(A)$ is a component of an identity index of a side of $s$ and hence is a
type of the macrotile containing $s$; the same holds for $B$ and $t$, and $\Phi(A)=\Phi(B)$ is
nonempty. If $A$ and $B$ are small quasi-squares, then each of them has exactly one special side,
so all members of $\Phi(A)$ except at most two are types of the macrotile of $s$, and likewise for
$B$; as the number of marked sides of a small quasi-square is large, this again forces the two
macrotiles to coincide. So in what follows $s$ and $t$ are types of one macrotile $\sigma'\alpha$.

Assume first that $A,B$ are original tiles with $l$ inner sides,
so that $|\Phi(A)|=|\Phi(B)|=2l$. By the definition of identity index,
 $\Phi(A)$ includes $s$ with multiplicity at least $l$, while $\Phi(B)$ includes $t$ with multiplicity at least $l$.
Thus $s\ne t$ implies that all identity indices on inner sides of $A,B$ are $\pair{s,t}$ or $\pair{t,s}$.
In other words, types $s,t$ share all inner sides, which is impossible provided $\tau^k$-macrotiles
are large enough. Indeed, in that case every side on the boundary of the region $s\cup t$ is an
outer side of the macrotile, so $s\cup t$ is the entire macrotile, that is, the macrotile consists
of two tiles only.

Now assume that $A,B$ are small quasi-squares. 
WLOG assume that $A$ is to the north of the central tile. Then $A$ has at least roughly $R/\Gamma$
northern sides. All of them except a constant number are shared with original tiles $C$
such that $B$ does not share any side with $C$. Each such side contributes to $\Phi(A)$ the type of
the corresponding tile $C$, and that type can occur in $\Phi(B)$ only among the at most two members
of $\Phi(B)$ that come from the special side of $B$, whose first index may be arbitrary. Thus
$\Phi(A)$ has approximately $R/\Gamma$
members that are outside $\Phi(B)$. As we have seen, we can make $R$ arbitrarily large, which implies a contradiction.
\end{proof}

Let us continue the proof of item (c). 
We showed that  tile $A$ has a legal parent $A_1$ of type $t$
and a normal parent $D$ also of type $t$.
We claim that $A_1=D$ and hence $D$ is legal.
As we have shown above, the tiles $A_1$ and $D$
have identical first four indices on all marked sides.
Their fifth indices on marked sides coincide as well, since $A_1$ and $D$ are normal tiles of the same type.
On non-marked sides all their indices are local and hence coincide for the same reason.
Therefore $A_1=D$, and so $D$ is legal. Lemma~\ref{l9} (Composition lemma) is proved.
  \end{proof}

Let us continue the proof of Proposition~\ref{th2}.
Let $\mathcal T$ be a proper tiling with legal tiles. 
Identity indices guarantee that it can be partitioned into $\sigma'$-macrotiles. For each of these macrotiles $\mathcal{M}$ there is a prototile $\alpha$ with $\pi(\mathcal{M})=\sigma'\alpha$.
Replace each macrotile $\mathcal M$ of the original tiling by the tile $D$ existing by 
the Composition Lemma.
  We obtain a $\sigma''$-composition $\mathcal T'$ of $\mathcal T$.
 Alignment indices on outer sides ensure that the composed tiling is  side-to-side.
Since $\mathcal T$ is proper, all indices of its tiles on adjacent ports match.
And since all indices on sides of composed tiles are borrowed from ports of $\mathcal T$,
the composed tiling $\mathcal T'$ is proper. 

  It remains to prove that each tile $D\in\mathcal T'$  is legal.
  By 
the Composition Lemma, $D$ is legal unless $D$ is of central
  acyclic type.
  We have to prove that  $D$ is legal even in that case.
  
Since all composed tiles are normal, $\mathcal T'$ can be again split into macrotiles.
Consider a macrotile $\mathcal M'\subset\mathcal T'$  containing a tile $D$ of a central acyclic  type $t$.
  Denote by $r$ the parent indices of tiles in $\mathcal M'$. By the Composition Lemma, the tiling $\mathcal M'$ has a composition 
$C$, which is a normal tile of type $r$,
see Fig.~\ref{f40}.
  \begin{figure}
    \begin{center}
      \includegraphics[scale=1]{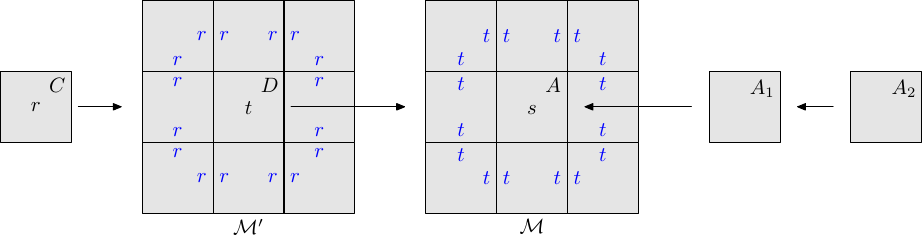}
    \end{center}
    \caption{Tiles $C,D,A$ have  types $r,t,s$, respectively, where $r$ is a non-central type and $t$
is a central acyclic type.
      The macrotile $\mathcal M'$ is a decomposition of $C$, and the macrotile $\mathcal M$ is a decomposition of $D$.
      The figure illustrates the proof that  $D$ is legal.}\label{f40}
  \end{figure}
   
Since the central child of $r$ is the acyclic type $t$,
the type $r$ is not central. Indeed, suppose $r$ were central. The form of a central type is a large quasi-square,
which is a fixed point of passing to the central child; hence $r$ and its central child $t$ have the same form.
Since the central child of a type depends only on its form, $t$ and $r$ have the same central child, namely $t$.
That is, the central child of $t$ is $t$ itself, so $t$ is cyclic --- contradicting the assumption that $t$ is acyclic.

Since $A$ is legal and is a grandchild of the normal tile $C$ of a non-central type, there exists a sequence of legal tiles
$$
A_0=A,A_1,A_2,\dots,A_n,
$$
in which each tile is a central child of the next one, the last tile is of a non-central type, and $n\in\{1,2\}$.
This is proven in exactly the same way as in the proof of item (c) of the Composition Lemma. (Only
the bound $n\in\{1,2\}$ is borrowed from there; the argument that excluded $n=2$ in item~(c) used
the non-centrality of the type of $D$, which is not the case now.)

Note that $A_1$ and $D$
are of the same form, as they have a common
central child $A$. Therefore $A_1$ is a central tile, hence $n=2$.
So, we have a chain of legal tiles $A=A_0,A_1,A_2$, in which each tile is a central child of the next one,
and $A_2$ is not central.

We claim that the tiles $A_1$ and $D$ have the same type. Exactly as in the proof of item~(c) of
the Composition Lemma, the tiles $A_1$ and $D$ are parents of one and the same tile $A$, hence
they are of the same form, hence the same print map $a\mapsto b(a)$ serves both of them, and
therefore
$$
\ind_{A_1}(j,a)=\ind_A(j,b(a))=\ind_D(j,a)
$$
for every marked side $a$ of their common form and every $j\le4$; in particular
$\Phi(A_1)=\Phi(D)$. This time, however, the tiles $A_1$ and $D$ are central, so
Lemma~\ref{l8} is not applicable to them. Instead we compare the multiplicities
of the type $t$ in these two multisets.

Being central tiles, both $A_1$ and $D$ have exactly $4N$ marked sides, and all of them are
special. The prints of the sides of $C$ are at most $Q$ of the marked sides of $D$, because the
form of the non-central type $r$ has at most $Q$ marked sides. On each of the remaining at least
$4N-Q$ special sides of $D$ the first index is the identity index of the corresponding side of the
central type $t$, and one of its two components equals $t$. Hence $t$ occurs in $\Phi(D)$ at least
$4N-Q$ times.

Assume now that the type $t'$ of $A_1$ differs from $t$, and count the occurrences of $t$ in
$\Phi(A_1)$. On a special side of $A_1$ that is not a print, the first index is the identity index
of a side of the central type $t'$; its components are $t'$ and the type of the small quasi-square
adjacent to that side, so neither of them is $t$ --- indeed $t'\ne t$ by the assumption, and $t$ is
a central type and hence is not a small quasi-square. The remaining marked sides of $A_1$ are
prints, and there are at most $Q$ of them; each side contributes at most two members to
$\Phi(A_1)$. So $t$ occurs in $\Phi(A_1)$ at most $2Q$ times. As $4N>3Q$ by the choice of $N$ in
Section~\ref{smarked}, we get $4N-Q>2Q$, which contradicts $\Phi(A_1)=\Phi(D)$. Therefore $t'=t$,
that is, $A_1$ and $D$ are of the same type.

Since  nonlocal indices of $A_1$ and $D$ coincide (they are inherited by their common child  $A$),
and their local indices coincide as well, being those of normal tiles of one and the same type,
the tiles $A_1$ and $D$ coincide, hence $D$ is legal.

Proposition~\ref{th2} is proved.
\end{proof}

Recall that to prove Theorem~\ref{th22}(a),
it remains to establish that for any proper tiling
$\mathcal T$ with tiles from $P'''$, the tiling
$\pi\Delta\mathcal T$ is a $\tau$-substitution tiling.

 We first show that all  crowns in $\pi\Delta\mathcal T$ are allowed.
Indeed,  let a vertex $V$ of a tile in  $\Delta\mathcal T$ 
 be given. By Proposition~\ref{th2} it has a $\sigma''$-composition
and hence can be  partitioned into $\sigma''$-macrotiles. 
Let $\mathcal M$ denote the macrotile that includes $V$.

Assume first that $V$ does not lie 
on the boundary of $\mathcal M$. 
In this case $V$ is an inner vertex of the $\tau$-macrotile $\pi\Delta\mathcal M$ 
and hence the crown centered at $V$ is allowed.

It remains to consider the case where $V$ lies 
on the boundary of $\mathcal M$. Since $\Delta$ does not produce new vertices on the boundary 
of $\sigma''$-macrotiles, $V$ is a vertex of a tile of $\mathcal M$.
On all sides incident to $V$ the second index contains some allowed crown, 
and that crown is the same for all the sides.
Indeed, let $A_1,\dots, A_n$ be all the tiles in $\mathcal T$ that include  $V$. And let 
$a_1,\dots, a_n$ be all their  sides that include $V$. 
Let $A_i$ have sides  $a_i$ and $a_{i+1 \bmod n}$.
 Then the crowns for $V$ in the second index of  
 $a_i$ in tiles $A_{i-1\bmod n}$ and  $A_i$ coincide, as $\mathcal T$ is proper.
On the other hand, the crowns for $V$ in $a_i$ and $a_{i+1 \bmod n}$ in the tile $A_i$ coincide, since 
$A_i$ is legal. 
 Moreover, for each tile $A_i$, the crown for $V$ in the second index of $a_i$ includes $A_i$.
 Therefore, the crown centered at $V$ can only be the one specified in the
 second indices, and therefore it is allowed.

Now we  show that, moreover, all finite fragments 
of $\pi\Delta\mathcal T$ occur in $\tau$-supertiles.
 By Proposition~\ref{th2} there exists a sequence $\mathcal T_0=\mathcal T, \mathcal T_m, \mathcal T_{2m},\dots$ of proper tilings   
 in which each tiling is a $\sigma''$-composition of the previous one.
 Then in the sequence 
$$
\pi \Delta \mathcal T_0=\pi\Delta\mathcal T, \pi\Delta\mathcal T_m, \pi\Delta\mathcal T_{2m},\dots
$$  
each tiling is a $\tau^m$-composition of the previous one.
As we have shown, all crowns in all these tilings are allowed.
 By Lemma~\ref{l22} the tiling $\pi\Delta\mathcal T_0$ is a substitution tiling. 
Proposition~\ref{th5} and hence Theorem~\ref{th22}(a) are proved.
 
 \subsection{Proof of Theorem~\ref{th22}(b)}

 Let $\tau$ denote the given substitution. Construct the tile 
 set $P'$ and the substitution $\sigma'=\Delta^{-1}\tau^m\Delta$
as before. Then we define $P''$ and $\sigma''$ along the same lines as before.
But this time we make a small modification:
\begin{itemize}
\item 
First we remove the crown index, say, set it to zero. 

\item
Second we change the definition of the alignment  index.
The reason for that is that  Proposition~\ref{th2}  guarantees only that a proper tiling $\mathcal T$
has a $\sigma''$-composition. This implies that   $\Delta\pi \mathcal T$ has a $\tau^m$-composition 
and hence has $\tau^i$-compositions for all $i=1,\dots,m$. However those compositions might not be
side-to-side tilings.

Now the alignment index on an outer side $b$
of a tile $B$ from macrotile $\sigma' A$ is defined as the tuple  $\pair{n_1,\dots, n_m}$
as follows. Let $a$ be the side of $A$ such that $b$ lies on the macroside $\sigma'a=\tau^m a$, and
for $i=0,1,\dots,m$ let $c_i$ denote the side of the supertile $\tau^i A$ that contains $b$; thus
$c_0=a$, $c_m$ is the side of $\tau^mA$ occupied by $b$, and $c_i\subseteq c_{i-1}$ for all $i$.
Then $n_i$ is the number of $c_i$, counted from left to right, among the sides of $\tau^iA$ into
which the side $c_{i-1}$ is subdivided.
Note that the numbering is \emph{relative} --- inside the side of the previous level --- and not
absolute along the whole superside $\tau^i a$.
In Fig.~\ref{f-8}  we have shown alignment indices on the 
north side  of a $\tau^{2}$-macrotile, where $\tau$ is the substitution of our first example (a square is substituted with a 3 by 3 grid).
\begin{figure}
  \begin{center}
    \includegraphics[scale=1]{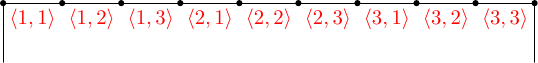}
  \end{center}
  \caption{Red indices on the north side of a $\tau^{2}$-macrotile, where $\tau$ is the substitution of our first example (a square is substituted with a 3 by 3 grid). }\label{f-8}
\end{figure}
This change makes Proposition~\ref{th2} hold in the following stronger form:
\emph{Every proper tiling $\mathcal T$ with legal tiles has a proper $\sigma''$-composition $\mathcal T'$ consisting of legal tiles.
  The tiling  $\mathcal T'$ has the following feature:
all tilings $\tau^{1}\Delta \pi(\mathcal T'), \dots, \tau^{m-1}\Delta \pi(\mathcal T')$
are side-to-side.}

Indeed, let $\mathcal T'$ be the tiling constructed in its proof.
Let $D$ be a tile from  $\mathcal T'$.  Since all $\tau$-supertiles are side-to-side tilings,
all supertiles $\tau^{m-1}\Delta\pi(D),\dots,\tau\Delta \pi(D)$ are side-to-side. So the problem may occur
only if for adjacent tiles $D,E\in  \mathcal T'$
there is $i<m$ such that some tiles $D'\in\tau^i \Delta\pi(D)$ and $E'\in\tau^i \Delta\pi(E)$
have overlapping sides $a',b'$ but do not coincide.

We show that this cannot happen. Consider the supersides $\tau^{m-i}(a')$
and $\tau^{m-i}(b')$. They belong to tiles of the tiling
$\tau^m\Delta\pi\mathcal T'=\Delta\pi\mathcal T$, and therefore they are
adjacent to each other side-to-side. Without loss of generality, we may
assume that the left endpoint of $\tau^{m-i}(a')$ lies to the right of,
or coincides with, the left endpoint of $\tau^{m-i}(b')$. Consider the
leftmost side $a''$ on the superside $\tau^{m-i}(a')$. Its alignment index
has the form $\pair{n_1,\dots,n_i,1,\dots,1}$: for every level $j>i$ the side of $\tau^jA$
containing $a''$ is the leftmost one in the subdivision of the side of the previous level, and it
is exactly here that the numbering has to be relative.

Since $\mathcal T$ is a
proper tiling, some side $b''$ is
adjacent to $a''$ side-to-side and has the same alignment index. This side lies on the superside
$\tau^{m-i}(b')$. Indeed, the sides $a'$ and $b'$ overlap along a segment of positive length, so
the left endpoint of $a''$ lies on $\tau^{m-i}(b')$ and differs from its right endpoint; and $a''$
cannot stick out beyond the right end of $\tau^{m-i}(b')$, since otherwise $a''$ would meet two
distinct sides of the side-to-side tiling $\Delta\pi\mathcal T$. By the
definition of the alignment index, it follows that $b''$ is the leftmost
side on the superside $\tau^{m-i}(b')$. Consequently, the left endpoints
of the supersides $\tau^{m-i}(a')$ and $\tau^{m-i}(b')$ coincide. The
coincidence of their right endpoints is proven analogously. Therefore
$a'$ and $b'$ are adjacent side-to-side.
\end{itemize}

The removal of crown indices makes Proposition~\ref{th1} hold for
$\tau$-side-to-side-hierarchical tilings: \emph{For any $\tau$-side-to-side-hierarchical tiling $\mathcal  T$ with tiles from $P$ there
is a proper tiling $\mathcal T '$ with legal tiles such that $\Delta\pi\mathcal T' =\mathcal  T$.}
The proof is the same, with one change: instead of appealing to Lemma~\ref{th11}, we take the chain
$\mathcal T_0=\mathcal T,\mathcal T_1,\mathcal T_2,\dots$ of side-to-side tilings provided by the
definition of a side-to-side-hierarchical tiling and pass to its subsequence with step $m$. The
positions inside the supertiles $\tau^iA$ that are needed to define the alignment tuples above are
read off from the intermediate members of that chain.

Now Theorem~\ref{th22}(b) follows. One direction is provided by the first modification: every
$\tau$-side-to-side-hierarchical tiling with tiles from $P$ is $\Delta\pi\mathcal T'$ for a proper
tiling $\mathcal T'$ with legal tiles. For the other direction, let $\mathcal T$ be any proper
tiling with legal tiles. Applying the strengthened form of Proposition~\ref{th2} repeatedly, we
obtain an infinite chain
$$
\mathcal T=\mathcal T^{(0)},\ \mathcal T^{(1)},\ \mathcal T^{(2)},\dots
$$
in which each tiling is a proper $\sigma''$-composition of the previous one and consists of legal
tiles. Passing to $\Delta\pi$ and inserting the intermediate levels
$\tau^1\Delta\pi\mathcal T^{(j+1)},\dots,\tau^{m-1}\Delta\pi\mathcal T^{(j+1)}$ between
$\Delta\pi\mathcal T^{(j)}$ and $\Delta\pi\mathcal T^{(j+1)}$, we obtain an infinite chain of
$\tau$-compositions of $\Delta\pi\mathcal T$.

All the tilings of this chain are side-to-side. For the intermediate levels this is exactly what the
strengthened form of Proposition~\ref{th2} says. For the levels $\Delta\pi\mathcal T^{(j)}$
themselves it follows from the local rule: the tiles of $\mathcal T^{(j)}$ are connected
side-to-side, the sides of a quasi-square are sides of the original tiles composing it, and inside a
quasi-square the original tiles form a fragment of a $\tau$-supertile, which is side-to-side by the
assumption STS. Hence $\Delta\pi\mathcal T$ is a $\tau$-side-to-side-hierarchical tiling, and
Theorem~\ref{th22}(b) is proved.

 \subsection{Proof of Theorem~\ref{th22}(c)}
 
 Obviously the families of $\tau$-hierarchical and $\tau^m$-hierarchical tilings coincide.\footnote{This is not the case for
 side-to-side-hierarchical tilings, and because of that we defined alignment indices in the proof of (b) in a more 
 complex way than in the proof of (a).}
So we can forget about $\tau$ and focus on $\tau^m$ instead, where $m$ is defined as in the proof of (a).
Moreover, a tiling $\mathcal T$ is $\tau^m$-hierarchical iff $\mathcal T=\Delta \mathcal T'$ for some $\sigma'$-hierarchical
tiling $\mathcal T'$. Using glues, we can forget about  $\tau^m$ and focus on $\sigma'$ instead.
That is, it suffices to prove that the family of $\sigma'$-hierarchical tilings is sofic.

To this end, we define the substitution $\sigma''$, that is, we color the tiles from $P'$, as in the proof of (a)
with the following modification.
As in the proof of item (b), the second index is not needed (it is always zero), so we can use only four
indices.
The alignment index of an outer side $b$ is now defined in the same way as in the proof of item (a):
If $b$ is an outer side of type $s$,
then the alignment index is defined as the number of the side in left-to-right order
along the macroside $\sigma' a$ that $b$ belongs to.
Otherwise, it is zero. 

The most important difference: now we sometimes set the fourth index on an outer side $b$ to zero,
and not to its alignment index.
We do this only in the case when $b$ lies on a macroside $\sigma' a$ such that $a$ is also an outer 
side whose fourth index is zero. 
On the other hand, on inner sides the fourth index will never be zero. To this end, we set
the fourth index  to the special value $\undef$
on all interior sides except special sides used for 
prints.

That is, in the definition of $\sigma'' A$ 
we change the requirement for the fourth index to the following ones:
\begin{itemize}
\item If $b$ lies on the macroside $\sigma' a$ where $a$  is an outer side
whose  fourth index is 0, then 
the fourth index of $b$ of $\mathcal M_s$
is $0$ (and not the alignment index of $b$, as before).
\item If $b$ is an outer side that does not fall into the previous case,
then the fourth index of $b$ of $\mathcal M_s$
is the alignment index of $b$, as before.
\item If $b$ is the print of $a$, then 
the fourth index of $b$ of $\mathcal M_s$ is equal to the fourth index of $a$, as before.
\item If $b$ is a boundary side, then  
the fourth index of $b$ of $\mathcal M_s$ is equal to the sixth index of  $A$, as before.
\item In the remaining cases the fourth index of $b$ of $\mathcal M_s$ is equal to $\undef$
(and not to 0, as before).
\end{itemize}

The family of legal tiles $P'''$ is defined in the same way as before.
We call the color of a side \emph{trivial} if its fourth and fifth indices are $0$.
The local rule for connecting tiles from $P'''$ is formulated as follows:
\begin{quote}
\emph{if the color of a side is nontrivial, 
then the side is adjacent side-to-side to a side of the same color of another tile}.
\end{quote}
Note that for legal tiles the color of all sides except outer sides
is non-trivial. And the color of an outer side is non-trivial iff its fourth
index is non-zero. Thus this local rule can be reformulated as follows: 
\begin{quote}
\emph{every side $a$ is adjacent side-to-side to a side of the same color of another tile
unless $a$ is an outer side with zero fourth index}.
\end{quote}
An important novelty: this rule allows two sides of trivial color to be connected with a shift.

\begin{prop}
A tiling $\mathcal T$ of the plane with tiles from $P'$
is $\sigma'$-hierarchical if and only if
$\mathcal T$ is obtained from some proper tiling 
with tiles from $P'''$ by erasing all indices. 
\end{prop}
\begin{proof}
First, we prove the possibility of proper decoration of all tiles of a $\sigma'$-hierarchical
tiling $\mathcal T$. This decoration is chosen in the same way as before
with one exception: fourth indices on non-marked outer sides are now not
always alignment indices.
More specifically, let a sequence of tilings $\mathcal T=\mathcal T_0,\mathcal T_1,\mathcal T_2,\dots$
witness that $\mathcal T$ is $\sigma'$-hierarchical. And let $a$ be an \emph{outer} side of a tile $A\in \mathcal T$.
And let $a_0=a,a_1,a_2,\dots $ be defined so that  $a_i$ is the side of  $A_i\in \mathcal T_i$ such that
$a_{i-1}$ lies on the macroside $\sigma' a_i$.

Case 1: If all sides $a_i$ are outer, then we set to 0
the fourth index of all $a_i$. Note that in this case the fifth index is 0 
by construction, hence the color of $a$ is trivial.\footnote{In this case
all other indices of $a$ are also 0.
Indeed, the first index is 0, as $a$ is outer. The second index is 0, as so is now the crown index.
And the third index is 0 unless  $a$ is a $j$th port for $j\le 5$. And even then it is 0:
for $j\le4$ because the third index of $a$ has no origin (an index can originate only in an inner side),
and for $j=5$ because it equals the fifth index of an outer parent side, which is 0. This observation
will not be used in the future.} 

Case 2: Otherwise, if some $a_i$ is an inner side, then  all $a_{i+1},a_{i+2},\dots$ are undefined and 
all $a_0,a_1,a_2,\dots,a_{i-1}$
are outer sides. In this case the fourth index of $a$ as well as  the fourth index of 
all $a_j$ is defined as its alignment index, as before. 

The resulting tiles are legal for the same reason as before.
However, the local rule must be verified anew.
If  $a$ is  an \emph{inner} side of a tile $A\in \mathcal T$, then the local rule for $a$
is verified exactly as before, as the decoration for $a$ and all sides overlapping with $a$
has not been changed (those sides must also be inner). 

Let   $a$ be an \emph{outer} side of a tile $A\in \mathcal T$.
In Case 1 the color of $a$ is trivial and hence the local rule is met by $a$.

In Case 2 let $b$ be a side of another tile $B\in\mathcal T$
that overlaps with $a$.
We will show that $a$ and $b$ are adjacent side-to-side and have matching colors.

Consider analogous sequences of tiles $B_1\in \mathcal T_1$,  $B_2\in\mathcal T_2$, $\dots$ and their sides $b_1,b_2,\dots$ for the side $b$. 
As each $b_j$ for $j\le i$ overlaps with $a_j$, the side  $b_j$
is outer for all $j<i$. For the same reason the side $b_i$ is inner.  Thus  $b$ also falls under Case 2 for the same $i$.
We are given that all $\sigma'$-supertiles are side-to-side and all $a_j,b_j$ belong to one supertile. 
Hence  $a_j=b_j$ for all $j\le i$. In particular $a=b$.
The fourth index on all $a_j,b_j$ 
is thus defined as before. Therefore, in this case the decoration is the same as before.
Hence $a$ and $b$ have matching colors.

In the other direction: we need to prove that after erasing the indices from the tiles
of a proper tiling $\mathcal T$ we obtain a $\sigma'$-hierarchical tiling. 
It is enough to prove that $\mathcal T$ is $\sigma''$-hierarchical. And for this it 
suffices to show that any proper tiling has a proper $\sigma''$-composition (an analog of Proposition~\ref{th2}):
\begin{quote} 
\emph{Every proper tiling $\mathcal T$ with legal tiles has a 
proper $\sigma''$-composition $\mathcal T'$ consisting of legal tiles.}
 \end{quote}
\begin{proof}
Every legal tile has an inner side, and on all its inner sides
the identity index is present. Therefore, in any
proper tiling, the tiles assemble into macrotiles in the same way as before, and the Composition Lemma
is proven in the same way as before. The fact that tiles can now meet each other not side-to-side does not interfere,
since this happens only for sides of trivial color, and such sides are always outer.

Let us show that the composed tiling $\mathcal T'$ satisfies the local rule.
Let $a$ be a side of $\mathcal T'$ with a nontrivial color. If $a$ is an inner side, then 
$\sigma'' a$ has not changed. Thus all sides on the macroside
$\sigma' a$ have alignment indices. Hence their colors
are nontrivial, and by the local rule for $\mathcal T$ they are adjacent side-to-side to 
other sides with matching alignment indices. Those sides are composed into a side $b$
that is adjacent to $a$ side-to-side. The side $b$ has matching color, as  $\mathcal T$ is proper.

Assume now that $a$ is an outer side  with non-zero fourth index.
Then by definition of $\sigma''$ the fourth indices of all sides on the macroside $\sigma' a$
are alignment indices, and we can repeat the argument.

Legality of all the resulting tiles is proven in exactly the same way as before.
\end{proof}
\end{proof}

\section{Under FLC assumption}

\begin{theorem}\label{th24}
For every  substitution $\tau$ that has FLC,  
(a)  the family of $\tau$-substitution tilings is sofic and (b) the family of $\tau$-hierarchical tilings is sofic.
\end{theorem}
\begin{proof}
(a) This statement can be derived from Theorem~\ref{th22}(a).

Let $P$ denote the given set of prototiles. 
First, we construct a set of prototiles $\tilde P$,
each of which  looks like a tile from $P$ with added
vertices (the angles at the added vertices are $180^\circ$).
Let $A$ be a tile of the form $\alpha\in P$ and let $A$ belong to a $\tau$-supertile
$\mathcal S=\tau^i\beta$, $\beta\in P$. 
Consider all tiles from $\mathcal S$ that share points with $A$. They form a
finite tiling, denoted by ${\mathcal S}_A$ and called the
\emph{neighbourhood of $A$ in $\mathcal S$}.

Let us show that the number of neighbourhoods is finite. Let a supertile $\mathcal S=\tau^i\beta$ be given,
and let $\mathcal S_A$ be the neighbourhood of tile $A$ in $\mathcal S$. Its diameter is at most some $D$ depending only
on the set of prototiles.
Consider the supertiles
$$
\tau^i\beta, \tau^{i-1}\beta,\dots, \beta.
$$
Each of them is a composition of the previous one.
The image of the set $\mathcal S_A$ under composition, that is, $\theta^{-1} \mathcal S_A$, has diameter $\theta$ times smaller than $\mathcal S_A$.
Therefore, for some $k$ depending only on the substitution, the diameter of the image of $\mathcal S_A$ under the $k$-fold
composition becomes smaller than $\eps$, which exists by Lemma~\ref{l22}. Assume first that
$i\ge k$ (the case $i<k$ is trivial, as then $A$ lies in a supertile of one of the finitely many
orders $0,1,\dots,k-1$, and there are finitely many possibilities for such $A$ altogether).
Hence this image
is covered by a single crown of the tiling
 $\tau^{i-k}\beta$. In other words, there exists a crown $C$ in the supertile $\tau^{i-k}\beta$ for which $\tau^k C$
 covers $\mathcal S_A$.
The crown $C$ can be chosen in finitely many ways, since $\tau$ has FLC.
Finally, the tiling $\mathcal S_A$ within $\tau^k C$
can also be chosen in finitely many ways.

Each pair $\pair{\alpha,\mathcal S_A}$ defines one prototile from the set $\tilde P$.
Geometrically, this tile coincides with the tile $\alpha_{\mathcal S}$, obtained from $\alpha$ by adding to each of its sides
all the vertices of the tiles from $\mathcal S_A$ that belong to that side. Prototiles of the form $\pair{\alpha,\mathcal S_A}$
will be called \emph{clones} of the prototile $\alpha$. 
The correspondence between
prototiles from $\tilde P$ and prototiles from $P$ is defined by the function $\tilde\rho$,
defined by the equality $\tilde\rho( \pair{\alpha,\mathcal S_A})=\alpha$.

Now we define the substitution $\tilde\tau$ on $\tilde P$. When applied to $\pair{\alpha,\mathcal S_A}$,
it acts like $\tau$, but
in the macrotile $\tau\alpha $, we replace each tile $B$ by one of its clones. This clone is 
$$
\pair{\text{the form of }B,(\tau\mathcal S_A)_B}=\pair{\text{the form of }B,(\tau\mathcal S)_B}.
$$
Since $\mathcal S$ is a supertile, the tiling $\tau\mathcal S$ is a
supertile as well. So this clone is in $\tilde P$.

Let us show that the substitution $\tilde\tau$ satisfies the STS property.
Let $\tilde\tau^i\pair{A,\mathcal S_A}$ be any $\tilde\tau$-supertile. 
We have to prove that $\tilde\tau^i\pair{A,\mathcal S_A}$ is a side-to-side tiling.
By the definition of $\tilde\tau$, we have
$$
\tilde\tau^i\pair{A,\mathcal S_A}=\{\pair{B,(\tau^i\mathcal S)_B}\mid B\in\tau^i A\}.
$$
Geometrically, the tile $\pair{B,(\tau^i\mathcal S)_B}$ coincides
with the tile $B$, to whose sides all the vertices of the tiles from
$\tau^i\mathcal S$ lying on this side have been added.
Therefore, replacing in the supertile $\tau^i \mathcal S$
each tile $B$ by its clone $\pair{B,(\tau^i\mathcal S)_B}$
yields a side-to-side tiling.

\begin{lemma}\label{l-tilde}
A tiling $\mathcal T$ with tiles from $P$ is a $\tau$-substitution tiling if and only if
$\mathcal T=\tilde\rho\tilde{\mathcal T}$ for some $\tilde\tau$-substitution
tiling with tiles from $\tilde P$.
\end{lemma}
We have moved the proof of this almost obvious lemma to the Appendix.

By Theorem~\ref{th22}(a),
the family of $\tilde\tau$-substitution tilings is side-to-side sofic. Denote by
$\hat P$ the corresponding set of colored prototiles, and by
$\hat\pi$ the color-erasing map.

Now we can define a coloring of tiles of $P$ and the local rule.
The number of colored versions of a prototile $\alpha\in P$ will be equal to the number of prototiles $\hat\alpha$ of $\hat P$ for which
$\tilde\rho\hat\pi\hat\alpha=\alpha$.
For each such prototile $\hat\alpha$, we mentally add additional vertices from its second component onto the sides of the prototile
$\alpha$. These imaginary vertices partition each side into several \emph{segments}. These segments are then
colored
as they are colored in $\hat\alpha$. The resulting sequence
of colored segments constitutes the color of the side. The resulting tile
is denoted by $\hat\xi\hat\alpha$.
The set of decorated tiles is denoted by $\bar P$,
and the mapping that erases colors and neighbourhoods is denoted by $\bar\pi$.
We define the local rule as follows: 
\begin{quote}
\emph{The  tiling must be segment-to-segment and 
colors of adjacent segments must match}.
\end{quote}

We need to prove that a tiling $\mathcal T$ with tiles from $P$ is a substitution tiling 
if and only if $\mathcal T=\bar{\pi}\bar{\mathcal T}$
for some tiling $\bar{\mathcal T}$ satisfying the local rule.
In one direction: 
assume that a tiling $\bar{\mathcal T}$ with tiles from $\bar P$
satisfies the local rule, that is, it is segment-to-segment and adjacent segments have the same color.
We need to prove that the tiling $\bar{\pi}\bar{\mathcal T}$ is a substitution tiling. Indeed,
the same tiling $\bar{\pi}\bar{\mathcal T}$ can be obtained by performing the following three steps (see the figure
where the blue color represents a neighbourhood):
\begin{center}
\includegraphics[scale=1.3]{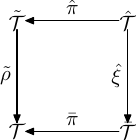}\hskip 2cm \includegraphics[scale=1.3]{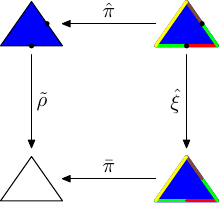}
\end{center}
(1) First, construct a proper tiling $\hat{\mathcal T}$ with tiles from the set $\hat P$,
for which $\hat\xi\hat{\mathcal T}=\bar {\mathcal T}$. To do this,  
break each side into separate  sides by adding additional vertices in those places where there is
a vertex of a tile from the neighbourhood. All the sides obtained are monochromatic. (2) Then erase the colors of sides, let 
$\tilde{\mathcal T}=\hat\pi\hat{\mathcal T}$ denote the 
obtained tiling. By the construction of $\hat P$ it is a $\tilde\tau$-substitution tiling.
(3) Erase the added vertices and neighbourhoods. By Lemma~\ref{l-tilde}, the resulting tiling $\mathcal T$
is a $\tau$-substitution tiling.

Conversely: given a $\tau$-substitution tiling $\mathcal T$, we have 
to construct a tiling $\bar{\mathcal T}$ satisfying the local rule,
for which $\mathcal T=\bar{\pi}\bar{\mathcal T}$.
First, using Lemma~\ref{l-tilde}
we construct a $\tilde\tau$-substitution tiling $\tilde{\mathcal T}$ such that
$\mathcal T=\tilde{\rho}\tilde{\mathcal T}$.
By the construction of $\hat P$, there exists a
proper tiling $\hat{\mathcal T}$ such that $\hat\pi\hat{\mathcal T}=\tilde{\mathcal T}$. By removing additional vertices from the tiles of this tiling,
we obtain the desired tiling $\bar{\mathcal T}=\hat\xi\hat{\mathcal T}$.
Indeed,
$
\bar\pi\bar{\mathcal T}=\bar\pi\hat\xi\hat{\mathcal T}=\tilde\rho\hat\pi\hat{\mathcal T}=\tilde{\rho}\tilde{\mathcal T}={\mathcal T}.
$

(b) Consider the same set of tiles $\tilde P$ as in the proof of item (a).
Each of its tiles is obtained from some tile of 
$P$ by adding new vertices. Consider 
the same substitution $\tilde\tau$, which extends $\tau$ and satisfies the STS property.
It is not difficult to prove that a tiling $\mathcal T$ with tiles from $P$ is $\tau$-hierarchical if and only if 
$\tilde\rho\tilde{\mathcal T}=\mathcal T$ for some  $\tilde\tau$-hierarchical tiling 
$\tilde{\mathcal T}$. 
Now we apply Theorem~\ref{th22}(c), 
obtaining a decoration of tiles from $\tilde P$ with the following property:
a tiling $\tilde{\mathcal T}$ with tiles from $\tilde P$ is $\tilde\tau$-hierarchical
if and only if for some proper tiling $\hat{\mathcal T}$ with colored 
tiles we have $\hat\pi \hat{\mathcal T}=\tilde{\mathcal T}$. 
Recall that the local rule for connecting tiles from $\tilde P$ reads:
\emph{if the color of a side is nontrivial, 
then it is adjacent side-to-side to another side of the same color}.

This local rule can be translated into a local rule
for the family of $\tau$-hierarchical tilings.
Indeed, in the transition from tiles of $P$ to tiles of $\hat P$, their sides 
are partitioned by the added vertices into segments, and each segment is colored in some color.
That is, different clones of the same prototile from $P$ differ from one another
in the added vertices and the colors of the segments into which the added vertices partition
the sides. Therefore, the local rule for connecting tiles from $\hat P$
translates into the following local rule for connecting tiles from $P$:
\emph{if the color of a segment is nontrivial, 
then it is adjacent segment-to-segment to another segment of the same color}.
\end{proof}

\section{Acknowledgments}

The author is sincerely grateful to Thomas Fernique for explaining the Fernique --- Ollinger technique in detail, 
to Nikita Andrusov, Andrei Romashchenko and Alexander Shen for listening to several previous versions of the results,
and to all participants of the Kolmogorov seminar at Moscow State University for attention and  patience.

\appendix
\section{Appendix}
\newcommand{\sector}{\text{sec}}
\newcommand{\angl}{\text{ang}}
\newcommand{\disc}{\text{disc}}

\begin{proof}[Proof of Lemma~\ref{l7}]
(a) Since $\tau$ has FLC, there are finitely many $\tau$-allowed crowns.
Call a crown \emph{complete} if the tiles of the crown surround its center $V$,
that is, the angles of these tiles at $V$ sum to $2\pi$.
For a complete $\tau$-allowed crown $K$ with center $V$, let $r(K)>0$ be the largest radius
such that the disk $\disc(V,r)$ of radius $r$ centered at $V$ is contained in the union of the tiles of $K$.
Let $d$ be the minimum of $r(K)$ over all complete $\tau$-allowed crowns $K$; since there are
finitely many of them, $d>0$.

Note that in a tiling $\mathcal T$ of the plane all of whose crowns are $\tau$-allowed every crown
is complete (its center is an inner vertex of $\mathcal T$). Hence for every vertex $V$ of $\mathcal T$
the disc  $\disc(V,d)$ is covered by the crown of $\mathcal T$ at $V$.

Next, let $\eta>0$ be the minimum, over all polygons $A_1\dots A_n$ in $P$, over all $i\in\{1,\dots,n\}$
and over all points $E\in[A_iA_{i+1\bmod n}]$ at distance at least $d/2$ from both ends of this segment,
of the distance from $E$ to the edge chain $A_{i+1\bmod n}A_{i+2\bmod n}\dots A_i$
(if there are no such points $E$, set $\eta=+\infty$). Finally, set $\eps=\min(\eta,d/2)$.

\emph{Case 1}.
If $S$ is covered by only one tile,
we are done --- any of its vertices can be taken as the center of the sought crown.

\begin{figure}
\begin{center}
\includegraphics[scale=1]{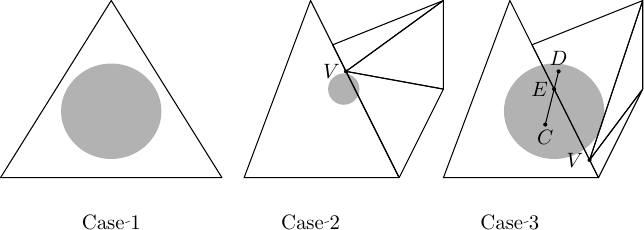}
\end{center}
\caption{The set $S$ is in grey.}
\end{figure}

\emph{Case 2}. If the distance from a point $E\in S$ to a vertex $V$ of a tile $A\in \mathcal T$ is smaller than $d/2$,
then the crown centered at $V$ covers $S$. Indeed, $V$ is a vertex of $\mathcal T$, so the crown at $V$
is complete and covers $\disc(V,d)$. Every point $F$ of $S$ satisfies
$|FV|\le|FE|+|EV|<\eps+d/2\le d$, so $S$ lies in that disk.

\emph{Case 3}. Otherwise $S$ cannot be covered by one tile from $\mathcal T$
and the distance from $S$ to vertices of $\mathcal T$-tiles is 
at least $d/2$.

Let $A$ be any tile intersecting $S$, say in point $C$, and let $D$
be any point from $S\setminus A$.
Consider the segment $[C,D]$. At some point $E$ that segment
leaves the tile $A$. W.l.o.g. assume that $S$ is convex and hence $E\in S$.
Since $E\in S$ and the diameter of $S$ is less than $\eps$, the whole set $S$ lies in $\disc(E,\eps)$.

The points of $[E,D]$ that lie very close to $E$ belong to some tile $B\ne A$ from $\mathcal T$.
That tile includes the point $E$. As $E$ is not a vertex of $\mathcal T$, it lies in the interior
of an edge shared by $A$ and $B$. We claim that
$S$ is covered by $A\cup B$.
Indeed, $E$ lies on an edge of $A$ (and on an edge of $B$) at distance at least $d/2$ from its ends,
so by the definition of $\eta$ every point of the border of the polygon $A\cup B$ is at distance
at least $\eta\ge\eps$ from $E$. As $S\subset\disc(E,\eps)$, all points of $S$ are in $A\cup B$.
It remains to notice that $A$ and $B$ belong to a crown of $\mathcal T$: indeed, either end of the
line segment shared by $A,B$ can be chosen as the center of that crown.

(b) Why does this argument fail in the case when $\mathcal T$ is a supertile?
There is a minor technical problem: for incomplete crowns $r(K)=0$. This is easily fixed as follows:
Define \emph{the angle of a crown} $K$ with center $V$, denoted by $\angl(K)$,  as the union of all
angles of the tiles $A\in K$ at $V$. For incomplete crowns it is natural to define $r(K)$ as the largest radius
$r$ such that the sector $\sector(K,r)=\disc(V,r)\cap\angl(K)$ is contained in the union of the tiles of $K$. And then
consider the sector instead of the disk.
\begin{figure}
\begin{center}
\includegraphics[scale=1]{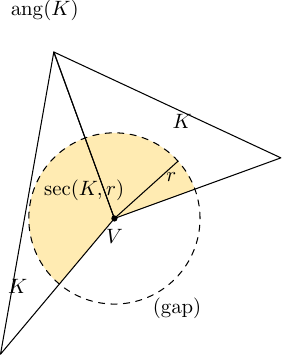}
\end{center}
\caption{An incomplete crown $K$ at $V$. The angle $\angl(K)$ is the union of the angles of the tiles
of $K$ at $V$, and $\sector(K,r)=\disc(V,r)\cap\angl(K)$ is the shaded sector.}\label{f-sector}
\end{figure}

However, this refinement does not solve two problems: in Case 2 the sector $\sector(K,d)$
may fail to cover $S\cap R$, because the set $R$ may contain points $d$-close to $V$ that lie
outside the angle $\angl(K)$ and are therefore not covered by the sector,
and in Case 3 the points of the segment $ED$ close to $E$ may lie outside $R$, so the tile $B$ is not defined.
Both problems can be solved by decreasing $\eps$.

Solving the first problem:
we need the set $R\cap \disc(V,d)$ to be included in $\angl(K)$
for any crown $K$ of the tiling $\mathcal T$ centered at $V$.
Ignoring the size, the set $R$ is some prototile.
If $V$ lies strictly inside this set, there is nothing to prove. Otherwise
$\angl(K)$ is the interior angle at some vertex $V$ of this prototile, or
a straight angle whose vertex $V$ lies strictly inside some side of this prototile.
Set $\nu$ equal to the minimum, over all polygons $A_1\dots A_n$ in $P$, over all $i\in\{1,\dots,n\}$
of the distance from $A_i$ to the edge chain $A_{i+1\bmod n}A_{i+2\bmod n}\dots A_{i-1\bmod n}$.
And set $\lambda$ equal to the minimum, over all polygons $A_1\dots A_n$ in $P$, of the distance
between two non-adjacent sides of $A_1\dots A_n$ (set $\lambda=+\infty$ if no polygon in $P$ has two non-adjacent sides).
Denote by $\mu$ the smallest exterior angle between adjacent sides of prototiles, that is, the smallest angle of the wedge lying
\emph{outside} the prototile between two of its adjacent sides (equivalently, $2\pi$ minus the interior angle at their shared vertex).
Finally, denote by $\gamma$ the minimal length of the sides of prototiles.

Then for any vertex $V$ on the boundary of the supertile $\mathcal T=\tau^j(A_1\dots A_n)$,
if one moves from it along a straight line
in the outward direction by a distance $x<\min(\nu,\lambda,\gamma\sin\mu)$, one leaves the supertile.
Indeed, suppose this is not the case, and denote by $G$ the first point of $R$
encountered during this movement. Consider two cases. (1) $V=\theta^j A_i$ for some $i\le n$.
Then $G$ belongs to the edge chain $\theta^j[A_{i+1\bmod n}A_{i+2\bmod n}\dots A_{i-1\bmod n}]$. And since we moved by less than $\nu\le\nu\theta^j$, this is impossible.
(2) $V$ lies on the interval $\theta^j (A_iA_{i+1\bmod n})$ for some $i\le n$.
But then $V$ is located at distance at least $\gamma$ from both ends of this interval.
If $G$ lies on the edge chain $\theta^j[A_{i+2\bmod n}A_{i+3\bmod n}\dots A_{i-1\bmod n}]$,
then $G$ and $V$ lie on non-adjacent supersides of $\tau^j(A_1\dots A_n)$; since this supertile is the polygon
$A_1\dots A_n$ scaled by $\theta^j$, we have $|VG|\ge\lambda\theta^j$, contradicting $x<\lambda\le\lambda\theta^j$.
Otherwise it lies on the segment
$\theta^j[A_{i+1\bmod n}A_{i+2\bmod n}]$ or $\theta^j[A_{i-1\bmod n}A_{i}]$, and then we get a contradiction with $x<\gamma\sin\mu$.
\begin{figure}
\begin{center}
\includegraphics[scale=.85]{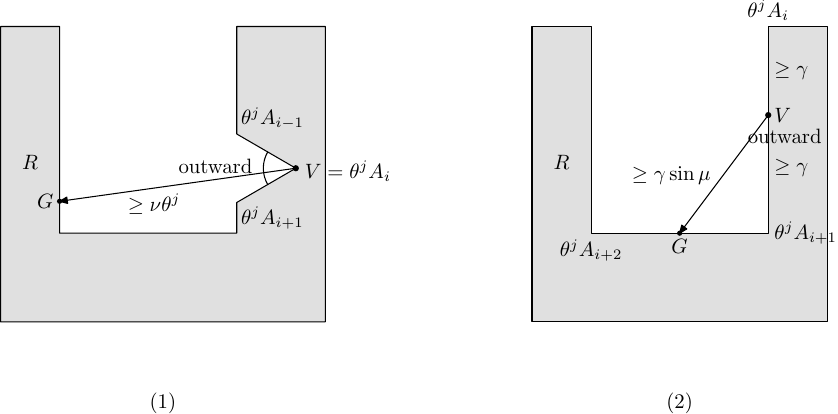}
\end{center}
\caption{Solving the first problem (the supertile $\mathcal T=\tau^j(A_1\dots A_n)$ may be non-convex):
moving outward from a boundary point by a distance less than $\min(\nu,\lambda,\gamma\sin\mu)$ leaves the supertile,
since otherwise the ray would re-enter $R$ at a point $G$ that is too far.
(1) $V=\theta^jA_i$ is a corner (here a reflex one, with exterior angle $60^\circ$); then $G$ lies on the far edge chain,
so $|VG|\ge\nu\theta^j$.
(2) $V$ lies in the interior of a superside $\theta^j[A_iA_{i+1}]$, hence at distance $\ge\gamma$ from both of its ends;
here the slanted outward ray re-enters at $G$ on the adjacent superside $\theta^j[A_{i+1}A_{i+2}]$, so $|VG|\ge\gamma\sin\mu$
(a re-entry on a non-adjacent superside would instead give $\ge\lambda\theta^j$).}\label{f-firstpr}
\end{figure}

Therefore, to solve the first problem it suffices to make $d\le \min(\nu,\lambda,\gamma\sin\mu)$.
Then $R\cap \disc(V,d)\subset\angl(K)$. Recall that in the second case we had
$S\subset \disc(V,d)$, hence
$$
S\cap R\subset \disc(V,d)\cap R\subset\angl(K)\cap \disc(V,d)=\sector(K,d)\subset K
$$
in Case 2.

Solving the second problem: we need that in Case 3 the points of the segment $ED$ close to $E$
lie in $R$. Suppose this is not the case. Consider the point $F$ at which the segment $ED$
returns to $R$. We know that the distance from both points
$F,E$ to any vertex of $\mathcal T$ is greater than $d/2$.
If the points $E,F$ belong to non-adjacent supersides of the supertile $\mathcal T=\tau^j(A_1\dots A_n)$, 
then $\eps\ge |EF|\ge\theta^j\lambda\ge\lambda$.
Otherwise, let $E$ belong to the superside $\theta^j[A_{i-1\bmod n}A_i]$, and $F$ to the superside  
$\theta^j[A_iA_{i+1\bmod n}]$ (for some $i$). Then
$\eps\ge |EF|> 2(d/2)\sin(\mu/2)$. Thus, if we define $\eps\le \min(\lambda,d\sin(\mu/2))$, both cases are impossible.
\begin{figure}
\begin{center}
\includegraphics[scale=1]{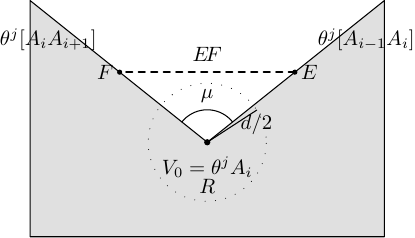}
\end{center}
\caption{Solving the second problem: $E$ and $F$ lie on the two supersides $\theta^j[A_{i-1}A_i]$ and $\theta^j[A_iA_{i+1}]$
meeting at the reflex vertex $V_0=\theta^jA_i$, both at distance $>d/2$ from $V_0$. The chord $EF$ crosses the
exterior wedge of angle $\mu$, so $|EF|>2(d/2)\sin(\mu/2)$.}\label{f-secondpr}
\end{figure}

From the above analysis it is clear that one can
define $\eps>0$ so as to satisfy
all the inequalities used: first decrease the number $d$ in the indicated way,
and then set $\eps=\min(\lambda, d\sin(\mu/2),\nu,\eta,d/2)$.
\end{proof}

\begin{proof}[Proof of Lemma~\ref{claim-prop1}]
Indeed, let a substitution tiling $\mathcal T$ be given.
By Proposition~\ref{th5}, there exists a proper tiling $\mathcal T'$ such that $\Delta\pi\mathcal T'=\mathcal T$. Consider the tiling
$\Delta\mathcal T'$.
For each tile $B$ of the tiling $\Delta\mathcal T'$ obtained by scattering
an added tile, and for each interior side of $B$, we add the corresponding glue to that side.
For the  resulting proper tiling $\tilde{\mathcal T}$ with tiles from $\tilde P$ we have
$\pi\tilde{\mathcal T}=\mathcal T$.
\begin{center}
\includegraphics[scale=1]{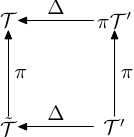}
\end{center}

Conversely, let $\tilde{\mathcal T}$ be a proper tiling
with tiles from $\tilde P$. We need to prove that the tiling $\pi \tilde{\mathcal T}$
is a substitution tiling.
In the tiling
$\tilde{\mathcal T}$, any tile with glue on at least one side
must be adjacent to a tile with the same glue. These tiles can
be assembled into an added tile. We perform this assembly and denote  by $\mathcal T'$ the resulting tiling with tiles from 
$P'$.
Since $\tilde{\mathcal T}$ is proper, the resulting tiling $\mathcal T'$ is also proper.
By Proposition~\ref{th5}, the tiling $\Delta\pi\mathcal T'$ is a substitution tiling.
It remains to note that $\Delta\pi\mathcal T'=\pi\Delta\mathcal T'=\pi \tilde{\mathcal T}$.
\end{proof}

\begin{proof}[Proof of Lemma~\ref{l-tilde}]
In one direction:
Let a $\tilde\tau$-substitution tiling $\tilde{\mathcal T}$ be given. We need to prove that
any finite fragment
$\mathcal F\subset\tilde\rho\tilde{\mathcal T}$
occurs in some $\tau$-supertile.
Obviously, $\mathcal F=\tilde\rho \tilde{\mathcal F}$ for some $\tilde{\mathcal F}\subset\tilde{\mathcal T}$. This fragment
$\tilde{ \mathcal F}$ belongs to some $\tilde\tau$-supertile $\tilde{\mathcal S}$. Therefore,
$\mathcal F$ is included in the $\tau$-supertile $\tilde\rho\tilde{\mathcal S}$.

Conversely, let $\mathcal T$ be a $\tau$-substitution tiling.
We construct the tiling $\tilde{\mathcal T}$ by adding to all tiles $A\in\mathcal T$ their
neighbourhoods in $\mathcal T$ (that is, all tiles from $\mathcal T$ that share a point with $A$). 
Since $\mathcal T$ is a $\tau$-substitution tiling, we obtain a tiling with tiles from the set $\tilde P$.

Let us show  that $\tilde{\mathcal T}$ is a $\tilde\tau$-substitution tiling.
Indeed, let $\tilde{\mathcal F}$ be any finite fragment of it. Consider any finite
patch $\tilde{\mathcal F}'$ of the tiling $\tilde{\mathcal T}$ that contains $\tilde{\mathcal F}$ strictly inside. This patch
is obtained from the patch $\tilde\rho\tilde{\mathcal F}'$ of the tiling $\mathcal T$ by adding
neighbourhoods in $\mathcal T$ to all tiles. Since $\mathcal T$ is a $\tau$-substitution tiling,
we know that $\tilde\rho\tilde {\mathcal F}'$ is included in some $\tau$-supertile $\mathcal S=\tau^i\alpha$.
Therefore, $\tilde\rho \tilde{\mathcal F}$ is strictly inside the supertile $\mathcal S=\tau^i\alpha$.

Notice that the pair 
$\tilde\alpha=\pair{\alpha,\{\alpha\}}$  is in $\tilde P$.
We claim that
$\tilde{\mathcal F}\subset\tilde \tau^i\tilde\alpha$.
Indeed, consider an arbitrary tile in $\tilde{\mathcal F}$.
By construction, it equals $\pair{\text{the form of }A, \mathcal T_A}$ for some tile $A$ lying strictly inside
$\tilde\rho\tilde{\mathcal F}'\subset\mathcal S$. Therefore, the neighbourhood of $A$ in the tiling $\mathcal T$ coincides with its neighbourhood in the tiling $\mathcal S$. That is, our tile is $\pair{\text{the form of }A, \mathcal S_A}$.

On the other hand, by the construction of the $\tilde\tau$ substitution, the second pairs
of all inner tiles in $\tilde\tau^i\tilde \alpha$ also consist of their neighbourhoods in
the tiling $\mathcal S$, so our tile belongs to $\tilde\tau^i\tilde \alpha$.
\end{proof}

\end{document}